\documentclass{article}

\usepackage{amsmath,amsthm,amssymb,bbm}
\usepackage[utf8]{inputenc}
\usepackage[margin=3cm]{geometry}
\usepackage{thmtools}
\usepackage{algorithm2e}
\usepackage{hyperref}
\usepackage{authblk}

\title{The Pareto cover problem} \date{}

\author{Bento Natura}
\affil{London School of Economics and Political Science}
\author{Meike Neuwohner}
\affil{Forschungsinstitut f\"ur Diskrete Mathematik, Universit\"at Bonn}
\author{Stefan Weltge}
\affil{Technische Universit\"at M\"unchen}

\declaretheorem{theorem}
\declaretheorem[sibling=theorem]{lemma}
\declaretheorem[sibling=theorem]{corollary}
\declaretheorem[sibling=theorem]{proposition}
\declaretheorem[style=remark,sibling=theorem]{claim}

\makeatletter
\def\renewtheorem#1{\expandafter\let\csname#1\endcsname\relax
    \expandafter\let\csname c@#1\endcsname\relax
    \gdef\renewtheorem@envname{#1}
    \renewtheorem@secpar
}
\def\renewtheorem@secpar{\@ifnextchar[{\renewtheorem@numberedlike}{\renewtheorem@nonumberedlike}}
\def\renewtheorem@numberedlike[#1]#2{\newtheorem{\renewtheorem@envname}[#1]{#2}}
\def\renewtheorem@nonumberedlike#1{  
    \def\renewtheorem@caption{#1}
    \edef\renewtheorem@nowithin{\noexpand\newtheorem{\renewtheorem@envname}{\renewtheorem@caption}}
    \renewtheorem@thirdpar
}
\def\renewtheorem@thirdpar{\@ifnextchar[{\renewtheorem@within}{\renewtheorem@nowithin}}
\def\renewtheorem@within[#1]{\renewtheorem@nowithin[#1]}
\makeatother

\theoremstyle{definition}
\renewtheorem{definition}[theorem]{Definition}

\newtheorem*{RestateTheoFPTAS}{Theorem~\ref{TheoFPTASGeneral}}

\usepackage{mathtools}
\usepackage{todonotes}
\usepackage{nicefrac}
\newcommand{\defn}{\coloneqq}

\newcommand{\ev}{\operatorname{E}}
\newcommand{\sz}{\operatorname{size}}

\newcommand{\Q}{\mathbb{Q}}
\newcommand{\R}{\mathbb{R}}
\newcommand{\Z}{\mathbb{Z}}
\renewcommand{\le}{\leqslant}
\renewcommand{\leq}{\leqslant}
\renewcommand{\ge}{\geqslant}
\renewcommand{\geq}{\geqslant}
\renewcommand{\t}{^{\intercal}}
\newcommand{\onevec}{\mathbf{1}}
\newcommand{\zerovec}{\mathbf{0}}

\newcommand{\eps}{\varepsilon}

\newcommand{\eul}{\mathrm{e}}
\newcommand{\mwp}[1]{\texttt{#1-WAY NUMBER PARTITIONING}}
\newcommand{\parti}{\texttt{NUMBER PARTITIONING}\space}

\begin{document}
\maketitle
\begin{abstract}
	We introduce the problem of finding a set $B$ of $k$ points in $[0,1]^n$ such that the expected cost of the cheapest point in $B$ that dominates a random point from $[0,1]^n$ is minimized.
	We study the case where the coordinates of the random points are independently distributed and the cost function is linear.
	This problem arises naturally in various application areas where customer's requests are satisfied based on predefined products, each corresponding to a subset of features.
	We show that the problem is NP-hard already for $k=2$ when each coordinate is drawn from $\{0,1\}$, and obtain an FPTAS for general fixed $k$ under mild assumptions on the distributions.
\end{abstract}

\section{Introduction}
Let $f \colon [0,1]^n \to \R$ be a cost function and $B \subseteq [0,1]^n$ be a finite set.
We consider the function
\[
	f_B \colon [0,1]^n \to \R \cup \{\infty\}, \quad f_B(x) \defn \min \{f(b) : x \le b, \, b \in B \},
\]
where we write $x \le b$ if $x$ is less or equal than $b$ in every coordinate and say that $b$ \emph{covers} $x$.
In other words, $f_B(x)$ is the smallest cost needed to cover $x$ with a point from $B$.
We say that $B$ is a \emph{Pareto cover}\footnote{In the context of multi-objective optimization, $x \le b$ is commonly referred to as $b$ \emph{Pareto-dominates} $x$.} of a probability measure $\mu$ on $[0,1]^n$ if a random point can be covered by at least one point from $B$ almost surely, i.e., $\mu(\{x \in [0,1]^n : x \le b \text{ for some } b \in B\}) = 1$.
Note that $B$ is always a Pareto cover if it contains the all-ones vector $\onevec$ (but for some probability measures, $B$ is not required to contain $\onevec$).

Given $f$, $\mu$, and an integer $k \ge 1$, we study the problem of finding a Pareto cover $B$ of $\mu$ with $|B| = k$ such that $\ev_\mu[f_B]$ is minimized.
That is, we are searching for a Pareto cover $B$ of predefined size such that the expected cost of covering a random point with a point from $B$ is smallest possible.

As an illustration, imagine a city with tourist attractions $[n] \defn \{1,\dots,n\}$ and suppose that the city wants to design $k$ books $B$ of vouchers for subsets of these attractions.
Each tourist $x$ will pick the cheapest book $b \in B$ that covers all attractions that $x$ wants to visit.
We can think of $x$ as a binary vector in $\{0,1\}^n$.
Assuming that we have some probability distribution over the tourists $x$, we want to determine $k$ books that, in expectation, cover the tourists requests in the cheapest way.
Note that, in this example, $\mu$ is a discrete measure on $\{0,1\}^n$.
Defining an appropriate cost function, an optimal Pareto cover of size $k$ is attained by a set of vectors in $\{0,1\}^n$, each corresponding to a book.

Similar applications can be given for other areas where customer's requests are satisfied based on predefined products, each corresponding to a subset of features.
Note that our model also allows for non-binary requests $x \in [0,1]^n \setminus \{0,1\}^n$, which may correspond to features that are available in different quality ranges.

For another application, imagine a gang of robbers that wants to steal paintings in an art gallery.
To this end, every gang member studies one painting $i \in [n]$ and estimates the probability $p_i$ of being able to steal it.
Their boss decides in advance which subset $S \subseteq [n]$ of paintings to steal.
If all corresponding gang members are successful (assuming that they act independently), then the gang will receive a value $v(S)$.
Otherwise, they all get caught and the gang receives $v(\emptyset)=0$.
The problem of finding a subset of paintings that maximizes the expected return can be modeled within the above framework as follows.
Let $\mu$ be the probability measure on $\{0,1\}^n$ that corresponds to setting each coordinate independently to $1$ with probability $1-p_i$, and set $f(x) := v([n])-v(\{i \in [n] : x_i = 0\})$ for all $x \in \{0,1\}^n$.
Denoting by $B^* = \{b^*, \onevec\}$ an optimal Pareto cover of size $k = 2$, it is easy to see that $S = \{i \in [n] : b_i^* = 0\}$ maximizes the expected return.

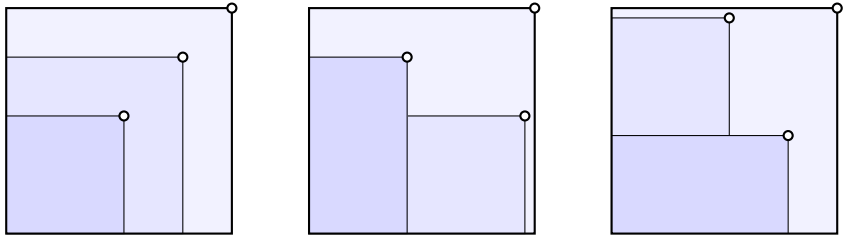
\begin{figure}
	\begin{center}
		\begin{tikzpicture}[scale=3]
			\fill[blue!5] (0,0) rectangle (1,1);
			\draw[fill=blue!10] (0,0.7826) -- (0.7826,0.7826) -- (0.7826,0) -- (0,0);
			\draw[fill=blue!15] (0,0.5217) -- (0.5217,0.5217) -- (0.5217,0) -- (0,0);
			\draw[thick] (0,0) rectangle (1,1);
			\draw[thick,fill=white] (0.5217,0.5217) circle (0.02);
			\draw[thick,fill=white] (0.7826,0.7826) circle (0.02);
			\draw[thick,fill=white] (1,1) circle (0.02);
		\end{tikzpicture}
		\hspace{2em}
		\begin{tikzpicture}[scale=3]
			\fill[blue!5] (0,0) rectangle (1,1);
			\draw[fill=blue!10] (0.4347,0.5217) -- (0.9565,0.5217) -- (0.9565,0) -- (0,0);
			\draw[fill=blue!15] (0,0.7826) -- (0.4347,0.7826) -- (0.4347,0) -- (0,0);
			\draw[thick] (0,0) rectangle (1,1);
			\draw[thick,fill=white] (0.4347,0.7826) circle (0.02);
			\draw[thick,fill=white] (0.9565,0.5217) circle (0.02);
			\draw[thick,fill=white] (1,1) circle (0.02);
		\end{tikzpicture}
		\hspace{2em}
		\begin{tikzpicture}[scale=3]
			\fill[blue!5] (0,0) rectangle (1,1);
			\draw[fill=blue!10] (0.5217,0.4347) -- (0.5217,0.9565) -- (0,0.9565) -- (0,0);
			\draw[fill=blue!15] (0.7826,0) -- (0.7826,0.4347) -- (0,0.4347) -- (0,0);
			\draw[thick] (0,0) rectangle (1,1);
			\draw[thick,fill=white] (0.7826,0.4347) circle (0.02);
			\draw[thick,fill=white] (0.5217,0.9565) circle (0.02);
			\draw[thick,fill=white] (1,1) circle (0.02);
		\end{tikzpicture}
	\end{center}
	\caption{For the Lebesgue measure (uniform distribution) on $[0,1]^2$ with cost $f(x_1,x_2) = x_1 + x_2$, the optimal Pareto covers of size $k = 3$ are $\{(\nicefrac{12}{23},\nicefrac{12}{23}),\, (\nicefrac{18}{23},\nicefrac{18}{23}),\, (1,1)\}$, $\{(\nicefrac{10}{23},\nicefrac{18}{23}),\, (\nicefrac{22}{23},\nicefrac{12}{23}),\, (1,1)\}$, and $\{(\nicefrac{18}{23},\nicefrac{10}{23}),\, (\nicefrac{12}{23},\nicefrac{22}{23}),\, (1,1)\}$.}
	\label{figFirstExample}
\end{figure}

Determining optimal Pareto covers of a given size is a difficult problem.
Finding an analytical solution seems to be non-trivial even for very basic probability measures and cost functions, see Figure~\ref{figFirstExample}.
In this work, we study the problem from the point of view of complexity theory.
We particularly focus on product measures and linear cost functions, for which the problem is already hard as our first results show.

Here, for a vector $p \in [0,1]^n$, let $\mu_p$ denote the discrete probability measure on $\{0,1\}^n$ where each coordinate is Bernoulli distributed with probability $p_i$, i.e.,
\begin{equation}
	\label{eqBernoulli}
	\mu_p(\{x\}) = \prod_{i:x_i=1} p_i \prod_{i:x_i=0}(1-p_i)
\end{equation}
for all $x \in \{0,1\}^n$.
Moreover, for a linear cost function $f$ given by $f(x) = c\t x$ for some $c \in \R^n$ and a finite set $B$, we overload our notation $c_B \defn f_B$.

\begin{theorem}
	\label{thmNPcomplete}
	Let $k \in \Z$, $k \ge 2$. Given $p \in ([0,1] \cap \Q)^n$, $c \in \Q^n$, $\gamma \in \Q$, the problem of deciding whether there is some Pareto cover $B$ of $\mu_p$ such that $|B| = k$ and $\ev_{\mu_p}[c_B] \le \gamma$ is weakly NP-complete for $k$ constant. Moreover, there are values of $k\in\Theta(n)$ for which it is strongly NP-hard.
\end{theorem}
If the size of the Pareto cover is part of the input, we do not know whether the corresponding problem is in NP.
In fact, computing the objective value of a single Pareto cover is already difficult:
\begin{proposition}
	\label{propSharpPHard}
	Given a Pareto cover $B$ for the uniform distribution $\mu$ on $\{0,1\}^n$, the problem of computing $\ev_\mu[\onevec_B]$ is \#P-hard.
\end{proposition}
On the positive side, for constant $k$, we derive a fully polynomial-time approximation scheme (FPTAS) for general product measures that satisfy some mild assumptions.
Essentially, we will consider products of \emph{nice} measures $(\mu_i)_{i=1}^n$ on $[0,1]$ that allow us to efficiently query an approximation of $\mu_i((a,b])$ for each $a,b,i$ as well as a positive lower bound on the expectation of the identity on $[0,1]$ with respect to each $\mu_i$.
A more formal definition will be given later.
\begin{theorem}
	\label{thmFPTAS}
	Let $k \in \mathbb{N}$ be fixed.
	Given a nice product measure $\mu$ on $[0,1]^n$ and $c \in \Q^n_{\ge 0}$, the problem of computing an (approximately) optimal Pareto cover of size $k$ with respect to $\mu$ and $c$ admits an FPTAS.\label{TheoFPTASGeneral}
\end{theorem}
Our paper is structured as follows.
In Section~\ref{secRelatedWork} we briefly discuss related work.
The proofs of Theorem~\ref{thmNPcomplete} and Proposition~\ref{propSharpPHard} are given in Section~\ref{secHardness}, where we also derive results for more general discrete distributions that will be used in our FPTAS.
The latter and hence the proof of Theorem~\ref{thmFPTAS} is presented in Section~\ref{secApproximation}.
We close with some open questions in Section~\ref{secQuestions}.

\section{Related Work}
\label{secRelatedWork}

To the best of our knowledge our setting for the general case $k>2$ has not been studied in the literature.
For the case $k=2$ similar problems have been studied in the area of \emph{stochastic optimization} in the context of \emph{chance constrained optimization}. Here one aims to find an optimal solution to a problem with stochastic constraints. A solution to the problem then needs to fulfill the constraints with probability $1 - \delta$ for some $\delta > 0$.

Linear chance constrained problems are of the form 
\begin{equation} \label{eq:chance_constrained_optimization}
\min \{\langle c, x \rangle : x \in X, \mathbb{P}_{\xi \sim \mu}[Ax \ge \xi] \ge 1 - \delta\}
\end{equation}
for a domain $X$, a distribution $\mu$, matrix $A$ and parameter $\delta$.
Note that our problem for $k=2$ with linear cost functions $f$ can be formulated as such a problem under the assumption that every Pareto cover has to contain $\onevec$. It remains to find the second second vector in the optimal Pareto cover, for which one can guess the probability that it covers an element drawn from $\mu$. This fits exactly into the framework of \eqref{eq:chance_constrained_optimization} where $A$ is the identity matrix and $X = [0,1]^n$.
For an overview on the topic, we refer to the work of Nemirovski and Shapiro~\cite{nemirovski2007convex} and Luedtke, Ahmed, and Nemhauser~\cite{luedtke2010integer}.
In principle, it is possible to extend this idea to the case $k \ge 3$, for instance by using techniques from~\cite{luedtke2010integer}.
However, it is unclear whether theoretically efficient (approximation) algorithms can be obtained by such an approach.

The Pareto cover problem has an interesting interpretation in the context of tropical geometry. It can be equivalently phrased as a problem of partitioning $[0,1]^n$ into $k$ regions $R_i$. For each region, one selects the \emph{tropical barycenter}, which is the coordinate-wise maximum of all its elements. This resembles a tropicalized version of classical Euclidean clustering algorithms and barycenters.

For randomized algorithms, tools from \emph{statistical learning theory} (sample complexity, coresets,\ldots) can be used to study the Pareto cover problem.
A natural approach would be to replace $\mu$ by a probability distribution $\tilde \mu$ that is a combination of polynomially many Dirac-distributions such that $\tilde \mu$ is ``close'' to $\mu$, and reduce the problem to finding an optimal Pareto cover for $\tilde \mu$.
In this work, however, we focus on deterministic algorithms and defer a discussion of such techniques to the full version of our paper.

\section{Hardness results}
\label{secHardness}

In this section, we show that finding optimal Pareto covers (of given sizes) is a computationally hard problem, even for simple discrete product measures and linear cost functions.
In particular, we prove Theorem~\ref{thmNPcomplete} and Proposition~\ref{propSharpPHard}.

In the first part, we focus on the case where $\mu$ is a discrete probability measure on $\{0,1\}^n$ given by some input vector $p \in [0,1]^n$ such that $\mu = \mu_p$, see~\eqref{eqBernoulli}.
We consider the following problem.

\begin{definition}
	The decision variant of the \emph{binary Pareto cover problem} is defined as follows:
	Given $p \in ([0,1] \cap \Q)^n$, $c \in \Q^n_{\ge 0}$, $\gamma \in \Q$, and $k \in \Z_{\ge 1}$, decide whether there is some Pareto cover $B$ of $\mu_p$ such that $|B| = k$ and $\ev_{\mu_p}[c_B] \le \gamma$.
\end{definition}

We assume that $p$, $c$, and $\gamma$ are given by their binary encodings.
We leave open how $k$ is encoded since in our applications it will be always polynomially bounded in $n$ (see Proposition~\ref{propSharpPHard}) or mostly even constant.

In Section~\ref{secHardnessConstantK}, we show that the binary Pareto cover problem is weakly NP-hard if $k$ is a fixed constant. In addition, we show that for $k=\frac{n+5}{3}$, the problem is strongly NP-hard.
For the case that $k$ is part of the input, we prove Proposition~\ref{propSharpPHard} showing that the problem of computing $\ev_{\mu_p}[c_B]$ for a given Pareto cover $B$ is \#P-hard in Section~\ref{secHardnessGeneralK}.
In view of this, it is unclear whether the binary Pareto cover problem is in NP if $k$ is part of the input.
However, for constant $k$, we establish in Section~\ref{secDiscreteInNP} that the problem is in NP, even for a more general class of discrete product measures. This result completes the proof of Theorem~\ref{thmNPcomplete} and plays an important role in the design of our approximation algorithm.

\subsection{NP-hardness}
\label{secHardnessConstantK}

In this section, we consider the binary Pareto cover problem.
Note that it can be solved efficiently in the case $k=1$ since then $B = \{b\}$ is an optimal solution, where $b \in \{0,1\}^n$ with $b_i = 1$ if and only if $p_i > 0$.
In the remainder, we show that the binary Pareto cover problem is  weakly NP-hard for $k=2$.
Similar, but slightly different proofs for the cases $k = 3$ and $k\geq 4$ can be found in Appendix~\ref{SecHardnessKequal3} and Appendix~\ref{SecHardnesskgeq4}, respectively. Moreover, our reduction for $k\geq 4$ proves strong NP-hardness for $k=\frac{n+5}{3}.$

In order to show that the binary Pareto cover problem is NP-hard for $k=2$, let us recall the \texttt{PARTITION} problem \cite{garey1990computers}, which is well-known to be (weakly) NP-hard: Given $a_1,\dots,a_n \in \Z_{\ge 1}$ with $\sum_{i=1}^n a_i$ even, decide whether there is a subset $I\subseteq [n]$ such that $\sum_{i\in I} a_i = \frac{1}{2} \sum_{i=1}^n a_i$.
We provide a reduction from \texttt{PARTITION} to the binary Pareto cover problem with $k=2$.

Before defining it precisely, let us describe the idea first.
Given a \texttt{PARTITION} instance $a_1,\dots,a_n$, set $\alpha \defn \frac{2}{a_1 + \dots + a_n}$ and consider an instance of the binary Pareto cover problem with $p_i = 1 - \eul^{-\alpha\cdot a_i}$ and $c_i = a_i$ for all $i \in [n]$.
Since all $c_i$ and $p_i$ are positive, every optimal Pareto cover of $\mu_p$ is of the form $B = \{b,\onevec\}$ where $b \in \{0,1\}^n$.
Setting $I \defn \{i \in [n] : b_i = 0 \}$, we see that the cost $\ev_{\mu_p}[c_B]$ of $B$ satisfies
\begin{align*}
	\sum_{i=1}^n a_i  - \ev_{\mu_p}[c_B] 
	&= \onevec\t c - \ev_{\mu_p}[c_B] \\
	& = \onevec\t c - c\t b \cdot \mu_p(\{x \in \{0,1\}^n : x \le b\}) - c\t \onevec \cdot \mu_p(\{x \in \{0,1\}^n : x \not \le b\})\\
	& = c\t (\onevec - b) \cdot \mu_p(\{x \in \{0,1\}^n : x \le b\}) \\
	& = \prod \nolimits_{i \in I} (1-p_i) \cdot \sum \nolimits_{i \in I} a_i \\
	& = \eul^{-\alpha \sum_{i \in I} a_i} \cdot \sum \nolimits_{i \in I} a_i \\
	& = h \left( \sum \nolimits_{i \in I} a_i \right),
\end{align*}
where $h \colon \R \to \R$, $h(x) = x \cdot \eul^{-\alpha\cdot x}$.
Since $h$ has its unique maximum at $x=\alpha^{-1}$, we see that $B = \{b,\onevec\}$ is a Pareto cover with $\ev_{\mu_p}[c_B] \le \sum_{i=1}^n a_i - h(\alpha^{-1})$ if and only if $I = \{i \in [n] : b_i = 0 \}$ satisfies $\sum_{i \in I} a_i = \alpha^{-1} = \frac{1}{2} \sum_{i=1}^n a_i$.
In other words, if $a_1,\dots,a_n$ is a ``yes'' instance for \texttt{PARTITION}, then there is a Pareto $B$ cover of size $k=2$ of $\mu_p$ with cost $\ev_{\mu_p}[c_B] \le \sum_{i=1}^n a_i - h(\frac{1}{2} \sum_{i=1}^n a_i)$.
If $a_1,\dots,a_n$ is a ``no'' instance, then every Pareto cover of size $k=2$ will have cost $\ev_{\mu_p}[c_B] > \sum_{i=1}^n a_i - h(\frac{1}{2} \sum_{i=1}^n a_i)$.

Unfortunately, the probabilities that we have used in the above argument cannot be polynomially represented.
However, we show that we can efficiently round them such that the above strategy still works.
More specifically, our probabilities and the threshold cost $\gamma$ will be defined as follows.

\begin{restatable}{lemma}{lemHardnessProbabilities} \label{lemHardnessProbabilities}
Given $a_1,\dots,a_n \in \Z_{\ge 1}$, we can compute $p_1,\dots,p_n \in [0,1]$ and $\gamma \in \Q$ in polynomial time such that
	\[
		\frac{1-\beta}{\eul^{\alpha \cdot a_i}} \le 1-p_i \le \frac{1+\beta}{\eul^{\alpha \cdot a_i}}
		\quad
		\text{for all } i \in [n],
		\text{and}
		\quad
		\frac{(1-\beta)^{n+2}}{\alpha \cdot \eul}
		\le \sum_{i=1}^n a_i - \gamma
		\le \frac{(1-\beta)^n}{\alpha \cdot \eul},
	\]
	where $\alpha \defn \frac{2}{\sum_{i=1}^n a_i}$ and $\beta \defn \frac{\alpha^2}{48(n+1)}$.
\end{restatable}
\begin{proof}
	See Appendix~\ref{SecAppendixHardnessk=2}.
\end{proof}

Let $p_1,\dots,p_n$ and $\gamma$ be given as in the above statement.
Note that we have $0 < p_i \le 1$ for all $i \in [n]$ since $p_i \le 1 - \frac{1-\beta}{\eul^{\alpha \cdot a_i}} \le 1$ and
\[
	p_i \ge 1 - \frac{1+\beta}{\eul^{\alpha \cdot a_i}}
	> 1 - \frac{1+\beta}{1 + \alpha \cdot a_i}
	\ge 1 - \frac{1+\beta}{1 + \alpha}
	> 0.
\]
Finally, set $c_i := a_i$ for all $i \in [n]$.
It remains to prove the following lemma.
\begin{lemma}
	\label{lemHardnessReduction}
	There is a subset $I \subseteq [n]$ with $\sum_{i\in I} a_i = \frac{1}{2} \sum_{i=1}^n a_i$ if and only if there is a Pareto cover $B$ of $\mu_p$ with $|B| = 2$ and $\ev_{\mu_p}[c_B] \le \gamma$.
\end{lemma}
\begin{proof}
	See Appendix~\ref{SecAppendixHardnessk=2}.
\end{proof}

Note that we have shown that the binary Pareto cover problem is weakly NP-hard for constant $k$.
We remark that, unless $P=NP$, the problem cannot be strongly NP-hard, as we derive an FPTAS in Section~\ref{secApproximation}.

\subsection{\#P-hardness}
\label{secHardnessGeneralK}

In the previous section we have seen that the binary Pareto cover problem is NP-hard, already for constant $k$.
The next natural question is whether the problem is in NP.
At first sight, a Pareto cover $B \subseteq \{0,1\}^n$ itself seems to be a canonical certificate for a ``yes`` instance.
However, with this choice, we should be able to compute the cost $\ev_{\mu_p}[c_B]$ efficiently.
Unfortunately, if $k$ is not part of the input, i.e., $B$ is not of constant size, computing $\ev_{\mu_p}[c_B]$ is hard.
More precisely, let us prove Proposition~\ref{propSharpPHard}, which states that, given a Pareto cover $B$ for the uniform distribution $\mu$ on $\{0,1\}^n$, the problem of computing $\ev_\mu[\onevec_B]$ is \#P-hard.

\begin{proof}[Proof of Proposition~\ref{propSharpPHard}]
	We use the fact that the problem of computing the number of vertex covers in a given undirected graph is \#P-hard~\cite{vadhan2001complexity}.
	Given a graph $G = (V,E)$, identify $V$ with $[n]$ and for every edge $e \in E$, let $b_e$ denote the characteristic vector of $V\setminus e$, the set of nodes that are not part of $e$.
	Let $\mu$ denote the uniform distribution on $\{0,1\}^n$, i.e., $\mu = \mu_p$ with $p = \frac{1}{2} \cdot\onevec$.
	Consider the Pareto cover $B \defn \{ b_e : e \in E \} \cup \{\onevec\}$.
	Setting $c = \onevec$, the cost of $B$ is equal to
	\begin{align*}
		\ev_\mu[c_B] & = (n-2) \cdot \mu(\{x \in \{0,1\}^n : x \le b_e \text{ for some } e \in E\}) \\
		& \quad \quad + n \cdot \mu(\{x \in \{0,1\}^n : x \not \le b_e \text{ for all } e \in E\}) \\
		& = (n-2) + 2 \cdot \mu(\{x \in \{0,1\}^n : x \not \le b_e \text{ for all } e \in E\}) \\
		& = (n-2) + \frac{|\{x \in \{0,1\}^n : x \not \le b_e \text{ for all } e \in E\}|}{2^{n-1}} \\
		& = (n-2) + \frac{|\{U \subseteq [n] : U \cap e \ne \emptyset \text{ for all } e \in E\}|}{2^{n-1}},
	\end{align*}
	and hence we see that $2^{n-1} \cdot (\ev_\mu[c_B] - (n-2))$ is the number of vertex covers in $G$.
\end{proof}

\subsection{Membership in NP for constant $k$}
\label{secDiscreteInNP}

For the binary Pareto cover problem, we have seen that computing the cost of a given Pareto cover $B$ is hard if $B$ can be of any size.
In this section, we show that the cost can be computed efficiently if $B$ is of constant size.
In fact, we prove that this is the case for a more general discrete version of our problem.

To this end, let us introduce the following notation.
Let $a = (a_0,\dots,a_{M+1})$ with $0 = a_0 < a_1 < \dots < a_M < a_{M+1} = 1$ and let $\mu_1,\dots,\mu_n$ be probability distributions on $\{a_0,\dots,a_{M+1}\}$.
We write $\mu = \prod_{i=1}^n \mu_i$ to denote the (discrete) probability distribution $\mu$ on $\{a_0,\dots,a_{M+1}\}^n$ given by $ \mu(x) = \prod_{i=1}^n \mu_{i}(\{x_i\}) $ for $x \in \{a_0,\dots,a_{M+1}\}^n$.
If $p = \left((p^i_j)_{j=0}^{M+1}\right)_{i=1}^n$ is such that $\mu_{i}(\{a_j\}) = p^i_j$, then we write $ \mu = \mu_{a,p} $.

We study the following problem.

\begin{definition}
	\label{defDiscretePareto}
	The decision variant of the \emph{discrete Pareto cover problem} is the following:
	Given $a = (a_0,\dots,a_{M+1})$ with $0 = a_0 < a_1 < \dots < a_M < a_{M+1} = 1$, $p = \left((p^i_j)_{j=0}^{M+1}\right)_{i=1}^n$ with $p^i_j \in [0,1]$ and $\sum_{j=0}^{M+1} p^i_j = 1$ for all $i$, $c \in \Q^n_{\ge 0}$, $\gamma \in \Q$, and $k \in \Z_{\ge 1}$, decide whether there is some Pareto cover $B$ of $\mu_{a,p}$ such that $|B| = k$ and $\ev_{\mu_p}[c_B] \le \gamma$.
\end{definition}

Again, we assume that $a$, $p$, $c$, and $\gamma$ are given by their binary encodings.
In our applications, $k$ will be always constant.

Note that it is easy to check whether a finite set $B \subseteq [0,1]^n$ is feasible for the above problem, i.e., that it is a Pareto cover of $\mu_{a,b}$.
In fact, let $x^* \in [0,1]^n$ be given by $x^*_i \defn \max \{a_j : p^i_j > 0 \}$ for $i \in [n]$.
Then $B$ is a Pareto cover of $\mu_{a,b}$ if and only if $|B| = k$ and $B$ contains at least one point that covers $x^*$.

Moreover, note that if $B$ is feasible for the above problem, then we may assume that $B \subseteq \{a_0,\dots,a_{M+1}\}^n$ holds since otherwise we may lower entries of points in $B$ without changing the set of points they cover and without increasing their cost.

\begin{proposition}
	\label{propDiscreteComputeCostPolytime}
	Let $k \in \Z_{\geq 1}$ be fixed.
	Given $a,p,c$ as in the discrete Pareto cover problem and a Pareto cover $B \subseteq \{a_0,\dots,a_{M+1}\}^n$ for $\mu = \mu_{a,p}$ with $|B| = k$, we can compute $\ev_\mu[c_B]$ in polynomial time.
\end{proposition}

Note that this shows that the discrete Pareto cover problem (and hence also the binary Pareto cover problem) is in NP if $k$ is constant.

In order to prove Proposition~\ref{propDiscreteComputeCostPolytime}, we make use of the following notation.
For vectors $b^1,\dots,b^k \in [0,1]^n$ and a vector $x \in [0,1]^n$ we define
\[
	J^i(x) \defn \left \{ j \in [k] : x_1 \le b^j_1, \dots, x_i \le b^j_i \right \}
\]
for $i \in [n]$, and set $J(x) \defn J^n(x)$.
Whenever we refer to $J(x),J^1(x),\dots,J^n(x)$, the vectors $b^1,\dots,b^k$ will be clear from the context.
Observe that for $J \subseteq [n]$, the set $\{x \in [0,1]^n : J(x) = J\}$ is a Borel set.

Note that $\{b^1,\dots,b^k\}$ is a Pareto cover for $\mu$ if and only if $\mu(\{x \in [0,1]^n : J(x) = \emptyset\}) = 0$.
Let us rephrase the cost of a Pareto cover using this new notation:

\begin{lemma}
	\label{lemObjectiveWithJs}
	Let $\mu$ be a probability measure on (the Borel $\sigma$-algebra on) $[0,1]^n$ and let $B = \{b^1,\dots,b^k\}$ be a Pareto cover of $\mu$.
	Then for every $c \in \R^n$ we have
	\[
		\ev_\mu[c_B] = \sum_{\emptyset \ne J \subseteq [k]} \mu(\{x \in [0,1]^n : J(x) = J\}) \cdot \min_{j \in J} c\t b^j.
	\]
\end{lemma}
\begin{proof}
	For $J \subseteq [k]$, note that $c_B(x) = \min_{j \in J} c\t b^j$ holds for all $x \in [0,1]^n$ with $J(x) = J$.
	The claim follows since the sets $(\{x \in [0,1]^n : J(x) = J\})_{J \subseteq [k]}$ are disjoint.
\end{proof}

Thus, in order to prove Proposition~\ref{propDiscreteComputeCostPolytime}, it suffices to show that we can compute the values $\mu_{a,p}(\{x\in [0,1]^n : J(x)=J\})$ for all $J \subseteq [k]$ in polynomial time.
To this end, we show how to iteratively compute the values $\mu_{a,p}(\{x\in [0,1]^n: J^i(x)=J\})$ for all $J \subseteq [k]$ and $i \in [n]$.
Lemma~\ref{LemInductionStart} takes care of the base case $i=1$, whereas Lemma~\ref{LemInductionStep} explains how to proceed from $i$ to $i+1$.

\begin{lemma}
	\label{LemInductionStart}
	Let $\mu_1,\dots,\mu_n$ be probability measures on $[0,1]$ and let $b^1,\dots,b^k \in [0,1]^n$.
	For $\mu = \prod_{i=1}^n \mu_i$ and $J \subseteq [k]$ we have
	\[
		\mu(\{x \in [0,1]^n : J^1(x) = J\}) = \mu_1\left( (\alpha,\beta] \cap [0,1] \right),
	\]
	where $\alpha = \max_{j \in [k] \setminus J} b_1^j$ and $\beta = \min_{j \in J} b_1^j$.
\end{lemma}

Here, we use the convention $\max \emptyset \defn -\infty$ and $\min \emptyset \defn +\infty$.

\begin{proof}[Proof of Lemma~\ref{LemInductionStart}]
	We have $J^1(x)=J$ if and only if $x_1\leq b_1^j$ for all $j \in J$ and $x_1 > b_1^j$ for all $j\in [k]\setminus J$.
	That is, $J^1(x) = \{ x \in [0,1]^n : x_1 \in (\alpha,\beta] \}$.
\end{proof}

\begin{lemma}
	\label{LemInductionStep}
	Let $\mu_1,\dots,\mu_n$ be probability measures on $[0,1]$ and let $b^1,\dots,b^k \in [0,1]^n$.
	For $\mu = \prod_{i=1}^n \mu_i$, $J \subseteq [k]$, and $i \in \{2,\dots,n\}$ we have
	\[
		\mu(\{x \in [0,1]^n : J^i(x) = J\})
		= \sum_{J \subseteq L \subseteq [k]} \mu(\{x \in [0,1]^n : J^{i-1}(x) = L\})\cdot \mu_i((\alpha_L,\beta_L] \cap [0,1]),
	\]
	where $\alpha_L = \max_{j \in L \setminus J} b_i^j$ and $\beta_L = \min_{j \in J} b_i^j$.
\end{lemma}
\begin{proof}
	The claim follows from the fact that $\{x \in [0,1]^n : J^i(x) = J\}$ is equal to
	\begin{align*}
		\bigcup_{J \subseteq L \subseteq [k]} \Big[ & \{x \in [0,1]^n : J^{i-1}(x) = L\} \\
		& \quad \cap \{ x \in [0,1]^n : x_i \le b_i^j \text{ for all } j \in J, \, x_i > b_i^j \text{ for all } j \in L \setminus J \} \Big]
	\end{align*}
	and the observation that the above sets are disjoint.
\end{proof}
Note that for measures $\mu = \mu_{a,p} = \prod_{i=1}^n \mu_i$ and $\alpha,\beta \in \Q \cup \{\pm \infty\}$, we can compute $\mu_i((\alpha,\beta] \cap [0,1])$ in polynomial time.
Moreover, for constant $k$, the sum in Lemma~\ref{LemInductionStep} only has a constant number of summands.
This yields Proposition~\ref{propDiscreteComputeCostPolytime}.

\section{Approximation algorithm}
\label{secApproximation}

In this part, we present an FPTAS (see~\cite{ausiello1999complexity}) for the Pareto cover problem with general product measures that satisfy some mild assumptions and constant $k$, see Theorem~\ref{thmFPTAS}.
More precisely, we provide an algorithm that receives an instance $I$ and a parameter $\gamma\in(0,1)\cap\mathbb{Q}$, and computes a $(1+\gamma)$-approximate solution to $I$ in time polynomial in $\gamma^{-1}$ and the encoding length of $I$.
Of course, we must first explain what exactly we mean by an instance $I$ and how the probability measures are encoded, which is done in Section~\ref{secGeneralParetoCoverProblem}.

Our algorithm proceeds in two steps:
First, it reduces a general instance $I$ to an instance $I_\gamma$ of the discrete Pareto cover problem (see Definition~\ref{defDiscretePareto}) with the property that any $(1+\frac{\gamma}{15})$-approximate solution $B$ to $I_\gamma$ constitutes a $(1+\gamma)$-approximate solution for $I$. The details of the reduction can be found in Section~\ref{SecGeneralToDiscrete}.
In a second step, we run an FPTAS that is designed for the discrete case, with input $I_\gamma$ and $\frac{\gamma}{15}$, which is presented in Section~\ref{SecRounding}.

\subsection{The (general) Pareto cover problem}
\label{secGeneralParetoCoverProblem}

In order to specify the instances our FPTAS is designed for, let us introduce the following notation.
Here, whenever we say that $\mu$ is a probability measure on $[0,1]^n$ we mean that $\mu:\mathfrak{B}([0,1]^n)\rightarrow[0,1]$ is a probability measure, where $\mathfrak{B}([0,1]^m)$ denotes the Borel $\sigma$-algebra on $[0,1]^m$. Given a probability measure $\mu$ on $[0,1]^n$, we say that a random variable $X$ is \emph{distributed according to $\mu$} and write $X \sim \mu$ if for any measurable set $S\subseteq [0,1]^n$, the probability that $X\in S$ equals $\mu(S)$. Finally, for probability measures $\mu_1,\dots,\mu_n$ on $[0,1]$ we define the product measure $\mu = \prod_{i=1}^n\mu_i:[0,1]^n\rightarrow[0,1]$ that is given by $\mu(I_1 \times \dots \times I_n)=\prod_{i=1}^n \mu_i(I_i)$, where $I_1,\dots,I_n \subseteq [0,1]$ are intervals.

We are ready to define the type of instances that our algorithm can process.

\begin{definition}[nice probability measure]
	We call a probability measure $\mu$ on $[0,1]$ \emph{nice}, if there is an oracle $\sigma:\{(a,b)\in\mathbb{Q}^2: -1\leq a < b \leq 1\}\times(0,1)\rightarrow \mathbb{Q}_{\geq 0}$ that, given three rational inputs $a<b\leq 1$ and $\delta\in(0,1)$, computes a rational number $\sigma(a,b,\delta)$ with \[(1+\delta)^{-1}\cdot\mu((a,b]\cap [0,1])\leq \sigma(a,b,\delta)\leq (1+\delta)\cdot\mu((a,b]\cap [0,1])\]
	in time polynomial in the encoding lengths of $a$, $b$ and $\delta$.
	\label{DefNicelyDistributed}
\end{definition}
Note that, as an implicit consequence of the bound on the computation time, we in particular require the encoding length of $\sigma(a,b,\delta)$ to be polynomial in the encoding lengths of $a$, $b$ and $\delta$.
Moreover, note that we deliberately do not request $a\geq 0$ in order to be able to query $\mu(\{0\})=\mu((-1,0]\cap[0,1])$.

\begin{definition}[Pareto cover problem]
	\label{defGeneralParetoCoveringProblem}
	The optimization variant of the \emph{(general) Pareto cover problem} is the following:
	Given oracles $\sigma_1,\dots,\sigma_n$ for nice probability measures $\mu_1,\dots,\mu_n$, $c \in \Q^n_{\ge 0}$, $k \in \Z_{\ge 1}$, and $\alpha \in (0,1) \cap \mathbb{Q}$ such that $\alpha \le \min_{i=1}^n \ev_{X_i \sim \mu_i}[X_i]$, find a Pareto cover $B$ of $\mu = \prod_{i=1}^n \mu_i$ with $|B| = k$ that minimizes $\ev_\mu[c_B]$.
\end{definition}

Given an instance $I$ of the Pareto cover problem, we refer to its binary encoding length by $\mathrm{size}(I)$.
To avoid confusion, we explicitly point out that $\mu_1,\dots,\mu_n$ are only part of the instance definition for the sake of the presentation and the analysis.
Any algorithm for the Pareto cover problem will only have access to them by querying the $\sigma_i$.
In particular, we define the encoding length of the oracles $\sigma_i$ as $0$.

Note that the above definition requires a positive lower bound $\alpha$ on the expected values $\ev_{X_i \sim \mu_i}[X_i]$, which is justified as follows. Consider an instance where $k=1$, $c\equiv 1$, and $\mu_i(\{2^{-2^n}\})=1$ for all $i$. The optimum solution is given by $b^1=2^{-2^n}\cdot\onevec$, which has an encoding length exponential in $n$, as does any polynomial factor approximation. Hence, we need some grasp of the order of magnitude of the values that are attained with high probability, to ensure that the encoding length of such a number is polynomial in the input size. We ensure this by requiring a lower bound on the expectations to be explicitly given.

\subsection{Reduction to discretized solutions}
\label{SecGeneralToDiscrete}

In this section, we want to reduce the general Pareto cover problem (see Definition~\ref{defGeneralParetoCoveringProblem}) to the (optimization version of the) discrete Pareto cover problem (see Definition~\ref{defDiscretePareto}). All proofs are deferred to Appendix~\ref{secAppendixGeneralToDiscrete}. Our goal is to prove the following statement:

\begin{restatable}{theorem}{TheoReduction}\label{TheoReduction}
Let an instance $I$ of the Pareto cover problem and a parameter $\gamma > 0$ be given. Then there is an instance $I_\gamma$ of the discrete Pareto cover problem with the property that any $(1+\frac{\gamma}{15})$-approximate solution to $I_\gamma$ constitutes a $(1+\gamma)$-approximate solution to $I$. Moreover, $I_\gamma$ can be computed from $I$ and $\gamma$ in time polynomial in $\gamma^{-1}$ and the encoding length of $I$. 
\end{restatable}
We do this in three steps: First, we show, that we can restrict ourselves to a discrete set of \emph{query coordinates} the coordinates of the elements of $B$ may originate from, provided we are willing to lose a factor of $(1+\frac{\gamma}{15})$ in the approximation guarantee. The number of these query coordinates will be polynomial in the size of $I$ and $\gamma^{-1}$. Next, we observe that once we know that the coordinates of the vectors in $B$ are selected from a discrete set, we do not need to know the precise distributions $(\mu_i)_{i=1}^n$ anymore, but only the probability with which a certain query coordinate is the minimum one large enough to cover coordinate $i$. This gives rise to discrete distributions, which we, however, cannot compute exactly as we only have approximate access to $(\mu_i)_{i=1}^n$. Hence, we take the probabilities the oracles provide and normalize them to obtain $I_\gamma$. Finally, we show that the costs of each solution with respect to $I$ and $I_\gamma$ only differ by $(1+\frac{\gamma}{15})$, which will conclude the proof.

We start by establishing some trivial, but useful lower bound on the objective value, which helps us to estimate how much we are allowed to ``lose'' during our reduction.
\begin{restatable}{proposition}{PropExpectationBoundsOpt} \label{PropExpectationBoundsOpt}
$\alpha\cdot\sum_{i=1}^n c_i\leq\ev_\mu[c_B].$
\end{restatable}

In order to reduce an instance of the generalized problem to a discrete instance, the main idea is to define a sufficiently dense set of \emph{query coordinates} and consider the discrete distributions we obtain from $(\mu_i)_{i=1}^n$ by ``shifting probability up to the next query coordinate''.
\begin{definition}[query coordinates]
	For $\epsilon\in(0,1)\cap\mathbb{Q}$ let $M = M_\epsilon \defn \lceil \log_{1+\epsilon} ((\epsilon\cdot\alpha)^{-1})\rceil$ and define the \emph{query coordinates $Q(\epsilon)$} to be the ascending sequence $q_0,\dots,q_{M+1}$ given by $q_0=0$, $q_i=\epsilon\cdot \alpha\cdot (1+\epsilon)^{i-1}$ for $i \in [M]$, and $q_{M+1}=1$.
	In a slight abuse of notation, we will write $x\in Q(\epsilon)$ as a shortcut for $x\in\{q_0,\dots,q_{M+1}\}$.\label{DefQueryPoints}
\end{definition}
\begin{restatable}{proposition}{PropSizeMeps} \label{PropSizeMeps}
$M_\epsilon\in\mathcal{O}(\log(\alpha^{-1})\cdot\epsilon^{-2})$.
\end{restatable}
In particular, we see that $M_\epsilon$ is polynomial in $\sz(I)$ and $\epsilon^{-1}$.
\begin{definition}[$\epsilon$-discrete solution]
	We call a Pareto cover $B$ of $\mu$ \emph{$\epsilon$-discrete} if $b \in Q(\epsilon)^n$ for all $b \in B$.
\end{definition}
The following lemma shows that there exist $\epsilon$-discrete solutions that are close to optimum for the instance we start with.
\begin{restatable}{lemma}{LemGoodDiscretizedSolution} \label{LemGoodDiscretizedSolution}
If $0<\epsilon<\frac{\gamma}{30}$, then there exists an $\epsilon$-discrete solution $B$ with $b_i > 0$ for all $b \in B$ and $i \in [n]$ of cost at most $(1+\frac{\gamma}{15})\cdot \mathrm{OPT}$.
\end{restatable}
As we are only interested in approximate solutions to the Pareto cover problem, we can restrict ourselves to $\epsilon$-discrete solutions. But given $\epsilon > 0$ and an $\epsilon$-discrete solution, we are not longer interested in the precise coordinates of a random vector $x$ sampled according to $(\mu_i)_{i=1}^n$, but only in the smallest query coordinate $q\in Q(\epsilon)$ that is no less than a certain coordinate of $x$. In particular, we obtain the same costs if instead of $(\mu_i)_{i=1}^n$, we consider the discrete distributions $(\bar{\mu}_i)_{i=1}^n$ that correspond to first sampling a random vector according to $(\mu_i)_{i=1}^n$ and then rounding every coordinate up to the next element of $Q(\epsilon)$ (see Figure~\ref{FigDiscretizeDistribution}). However, we cannot actually determine $(\bar{\mu}_i)_{i=1}^n$ exactly since we only have approximate access to $(\mu_i)_{i=1}^n$ via the oracles $(\sigma_i)_{i=1}^n$. Hence, we have to settle for an approximation $(\tilde{\mu}_i)_{i=1}^n$ that we obtain by querying the oracles and normalizing the obtained probability values. The following lemma shows that going from $(\mu_i)_{i=1}^n$ to the discrete distributions $(\tilde{\mu}_i)_{i=1}^n$ does not change the cost of any $\epsilon$-discrete solution by much.
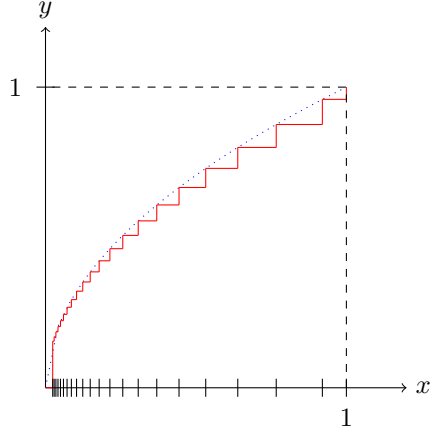
\begin{figure}[t]
\begin{center}
\begin{tikzpicture}[scale = 4]
\draw[->] (0,0) -- (1.2,0) node[right] {$x$};
\draw[->] (0,0) -- (0,1.2) node[above] {$y$};
\draw[domain=0:1, blue, dotted] plot ({\x}, {sqrt(\x)});
\draw[red] ({0},{0})--({0.02*1.2},{0});
\draw[red] ({0.02*1.2},{0})--({0.02*1.2},{sqrt(0.02*1.2)});
\foreach \x in {1,...,20}
{
	\draw[red] ({0.02*1.2^\x},{sqrt(0.02*1.2^\x)})--({0.02*1.2^(\x+1)},{sqrt(0.02*1.2^\x)});
	\draw[red] ({0.02*1.2^(\x+1)},{sqrt(0.02*1.2^\x)})--({0.02*1.2^(\x+1)},{sqrt(0.02*1.2^(\x+1))});
	\draw (0.02*1.2^\x,-0.03)--(0.02*1.2^\x,0.03);
}
\draw (0.02*1.2^21,-0.03)--(0.02*1.2^21,0.03);
\draw[red] ({0.02*1.2^21},{sqrt(0.02*1.2^21)})--({1},{sqrt(0.02*1.2^21)});
\draw[red] ({1},{sqrt(0.02*1.2^21)})--({1},{1});
\draw (-0.03,1)--(0.03,1);
\node at (-0.1,1) {$1$};
\draw (1,-0.03)--(1,0.03);
\node at (1,-0.1) {$1$};
\draw[dashed] (1,0)--(1,1)--(0,1);
\end{tikzpicture}
\end{center}
\caption{The functions $x\mapsto \mu_i([0,x])$ (blue, dotted) and $x\mapsto \bar{\mu}_i([0,x])$ (red).}\label{FigDiscretizeDistribution}
\end{figure}
\begin{restatable}{lemma}{LemSimilarObjective}  \label{LemSimilarObjective}
Let $\epsilon \defn \frac{\gamma}{60n}$ and let $q_0,\dots,q_{M_\eps+1}$ denote the query coordinates $Q(\epsilon)$.
Letting $q_{-1} \defn -1$, we define discrete probability measures $\tilde \mu_1,\dots,\tilde \mu_n$ by
\[
	\tilde{\mu}_i(\{q_\ell\}) = p^i_\ell \defn \frac{\sigma_i(q_{\ell-1},q_\ell,\epsilon)}{\sum_{t=0}^{M_\epsilon +1}\sigma_i(q_{t-1},q_t,\epsilon)}.
\]
Every $\epsilon$-discrete Pareto cover $B$ of $\prod_{i=1}^n \tilde{\mu}_i $ is also a Pareto cover of $\mu = \prod_{i=1}^n \mu_i$ and vice versa, and moreover, we have
\[\left(1+\tfrac{\gamma}{15}\right)^{-1}\cdot\ev_\mu[c_B]\leq \ev_{\prod_{i=1}^n \tilde{\mu}_i}[c_B]\leq \left(1+\tfrac{\gamma}{15}\right)\cdot\ev_\mu[c_B].\]
\end{restatable}

With the above statement, we are ready to prove Theorem~\ref{TheoReduction}.

\begin{restatable}{corollary}{CorReductionFPTAS} \label{CorReductionFPTAS}
Given an FPTAS for the discrete Pareto cover problem, we obtain an FPTAS for the (general) Pareto cover problem by computing, given an instance $I$ for the general problem and $\gamma\in(0,1)\cap\mathbb{Q}$, the instance $I_\gamma$ and then applying the FPTAS to $I_\gamma$ with parameter $\frac{\gamma}{15}$. 
\end{restatable}

\subsection{FPTAS via rounding}
\label{SecRounding}

The goal of this section is to develop an FPTAS for the discrete Pareto cover problem.
All proofs are deferred to Appendix~\ref{secAppendixRounding}.
Let $a,p,c,k$ define an instance of the discrete Pareto cover problem, where $k$ is a constant.
For every $i \in [n]$ let $l_i\defn\max\{l: p^i_l >0\}$ and define $a^* \defn (a_{l_1},\dots,a_{l_n})$.
Recall that every Pareto cover $B$ of $\mu = \mu_{a,p}$ must contain a point that covers $a^*$.
Conversely, every finite set $B \subseteq [0,1]^n$ containing $a^*$ is a Pareto cover of $\mu$.
Since the costs are non-negative, we can restrict ourselves to Pareto covers that contain $a^*$.

Next, we discuss how to determine a cover of approximately minimum cost.
Recall that in Section~\ref{secDiscreteInNP}, Lemma~\ref{lemObjectiveWithJs}, given a Pareto cover $B = \{b^1,\dots,b^k\}$, we have seen that we can express our objective as
\[
	\ev_\mu[c_B] = \sum_{\emptyset\neq J\subseteq[k]} \mu(\{x\in [0,1]^n: J(x)=J\}) \cdot \min_{j\in J} c(b^j).
\]
Even more, we know that we can iteratively compute the values $\mu(\{x\in [0,1]^n: J^i(x)=J\})$ for $i\in[n]$ and $J \subseteq [k]$ in polynomial time, see Lemma~\ref{LemInductionStep}. In doing so, the only information we need to proceed from $i$ to $i+1$ are the probabilities $\mu(\{x\in [0,1]^n: J^i(x)=J\})$ for $J \subseteq [k]$ and the values $(b^j_{i+1})_{j=1}^k$, but \emph{no} further information on the values $b^j_l$ for $l\in [i]$ and $j\in[k]$. What is more, by definition, the values $\mu(\{x\in [0,1]^n: J^i(x)=J\})$ for $J \subseteq [k]$ and $\sum_{l=1}^i c_l\cdot b^j_l$ for $j \in [k]$ do not depend on the coordinates $b^j_l$, $l=i+1,\dots,k$, $j\in[k]$. All in all, these are best preconditions for a dynamic programming approach and motivate the following definition:

\begin{definition}
	A \emph{candidate} is a triple $(i,(P_J)_{J\subseteq [k]}, (C_j)_{j=1}^k)$.
	We call such a candidate \emph{valid} if there exists a Pareto cover $B = (b^j)_{j=1}^k$ with $b^k=a^*$ such that 
	\begin{align*}
		C_j & = \sum \nolimits_{l=1}^i c_l\cdot b^j_l \text{ for all } j \in [k] \text{ and} \\
		P_J & =\mu(\{x\in [0,1]^n: J^i(x)=J\}) \text{ for all } J \subseteq [k].
	\end{align*}
	We say that $B$ \emph{witnesses} the validity of the candidate.
	The \emph{cost} of a candidate is given by
	$\mathrm{Cost}(\mathcal{C}):=\sum_{\emptyset\neq J\subseteq [k]} P_J\cdot\min_{j\in J} C_j.$
\end{definition}

Note that the definition of the cost of a candidate is in accordance with Lemma~\ref{lemObjectiveWithJs}.

\begin{definition}
	For a Pareto cover $B=(b^j)_{j=1}^k$ with $b^k=a^*$ and $i \in [n]$ let
	\[
		\mathrm{Cand}^i(B) \defn (i,(P_J)_{J\subseteq [k]}, (C_j)_{j=1}^k),
	\]
	where $P_J=\mu(\{x\in [0,1]^n: J^i(x)=J\})$ for $J \subseteq [k]$ and $C_j =\sum_{l=1}^i c_l\cdot b^j_l$ for $j \in [k]$.\label{DefCandForCover}
\end{definition}
A naive approach to tackle the discrete Pareto cover problem would now be to iteratively enumerate all valid candidates for $i=1,\dots,n$, select a candidate for $i=n$ that yields the minimum objective value, and then back-trace to compute a corresponding cover. The problem with this idea is of course that we do not have a polynomial bound
on the number of candidates we generate. To overcome this issue, we round the candidates appropriately to ensure a polynomial number of possible configurations, whilst staying close enough to the original values to obtain a good approximation of the objective for $i=n$. Observe that for constant $k$, the number of entries of each candidate is constant, which means that it suffices to polynomially bound the number of values each of them may attain.

To this end, consider Algorithm~\ref{AlgDPValidCandsRounded}. The gray lines are not part of the algorithm itself, but only needed for its analysis. 
\begin{algorithm}
	\DontPrintSemicolon
	\KwIn{$(a_l)_{l=0}^{M+1}$, $((p^i_l)_{l=0}^{M+1})_{i=1}^n$, $(c_i)_{i=1}^n$, $k$, $\epsilon\in\mathbb{Q}\cap(0,1)$}
	\KwOut{a table $\mathcal{T}$ of rounded candidates}\;
	For $i=1,\dots,n$ compute $l_i:=\max\{l: p^i_l>0\}$.\;
	$a^*\gets (a_{l_i})_{i=1}^n$,	$A\gets\{a_0,\dots,a_{M+1}\}$,$\mathcal{T}\gets\emptyset$\;
	\textcolor{gray}{$\mathrm{AllWits}(-)\gets\emptyset$}\;
	$\delta\gets\frac{\epsilon}{4n}$\;
	\ForEach{$(\beta^j)_{j=1}^k\in A^k$ with $\beta^k=a^*_1$\label{LineForeachBeta1}}{
		Define $(P_J)_{J\subseteq [k]}$ by $P_J\gets\mu_1((\max_{j\in[k]\setminus J} \beta^j,\min_{j\in J} \beta^j]\cap [0,1])$\label{LineStartInner1}\;
		$P_J\gets\begin{cases}
		(1+\delta)^{\lfloor\log_{1+\delta}P_J\rfloor} &, P_J >0\\
		0 &, P_J=0
		\end{cases}$\;
		Define $(C_j)_{j=1}^k$ by $C_j\gets \begin{cases}
		(1+\delta)^{\lfloor\log_{1+\delta}(c_1\cdot \beta^j)\rfloor} &, c_1\cdot \beta^j >0\\
		0 &, c_1\cdot \beta^j=0
		\end{cases}$\;
		$\mathcal{T}\gets\mathcal{T}\cup\{(1,(P_J)_{J\subseteq [k]}, (C_j)_{j=1}^k)\}$\;
		$b^j_1\gets\beta^j$, $j=1,\dots,k$, $b^j_i\gets 0$, $i=2,\dots,n$, $j=1,\dots,k-1$, $b^k_i\gets a^*_i$, $i=2,\dots,n$\;
		$\mathrm{Witness}((1,(P_J)_{J\subseteq [k]}, (C_j)_{j=1}^k))\gets (b^j)_{j=1}^k$\label{LineEndInner1}\;
		\color{gray}$\mathrm{AllWits}((1,(P_J)_{J\subseteq [k]}, (C_j)_{j=1}^k))\gets \mathrm{AllWits}((1,(P_J)_{J\subseteq [k]}, (C_j)_{j=1}^k))\cup\{(b^j)_{j=1}^k\}$\;\color{black}}
	\For{$i=2$ \textbf{to} $n$}{
		\ForEach{$(i-1,(P^{i-1}_J)_{J\subseteq [k]},(C^{i-1}_j)_{j=1}^k)\in\mathcal{T}$\label{LineForLoopOldCands}}{
			\ForEach {$(\beta^j)_{j=1}^k\in A^k$ with $\beta^k=a^*_i$\label{LineForBeta2}}{
				Define $(P^i_J)_{J\subseteq [k]}$ by
				$P^i_J\gets \sum_{\substack{L:\\ J\subseteq L\subseteq[k]}}P^{i-1}_L\cdot\mu_i((\max_{j\in L\setminus J}\beta^j,\min_{j\in J} \beta^j]\cap [0,1])$\label{LineStartInner2}\;
				$P^i_J\gets  \begin{cases}
				(1+\delta)^{\lfloor\log_{1+\delta}P^i_J\rfloor} &, P^i_J >0\\
				0 &, P^i_J=0
				\end{cases}$\;
				Define $(C^i_j)_{j=1}^k$ by $C^i_j\gets C^{i-1}_j+c_i\cdot \beta^j$\;
				$C^i_j\gets \begin{cases}
				(1+\delta)^{\lfloor\log_{1+\delta} C^i_j\rfloor} &, C^i_j > 0\\
				0 &, C^i_j = 0
				\end{cases}$\;
				$\mathcal{T}\gets \mathcal{T}\cup\{(i,(P^i_J)_{J\subseteq[k]},(C^i_j)_{j=1}^k)\}$\;
				$(b^{i-1,j})_{j=1}^k\gets\mathrm{Witness}((i-1,(P^{i-1}_J)_{J\subseteq [k]},(C^{i-1}_j)_{j=1}^k))$\;
				Define $(b^{i,j})_{j=1}^k$ by $b^{i,j}_l:=\begin{cases} b^{i-1,j}_l &, l \neq i\\ \beta^j &, l=i\end{cases}$\;
				$\mathrm{Witness}((i,(P^i_J)_{J\subseteq[k]},(C^i_j)_{j=1}^k))\gets (b^{i,j})_{j=1}^k$\label{LineEndInner2}\;
				\color{gray}\ForEach{$(\tilde{b}^{i-1,j})_{j=1}^k\in \mathrm{AllWits}((i-1,(P^{i-1}_J)_{J\subseteq [k]}, (C^{i-1}_j)_{j=1}^k))$ \label{LineForLoopAllWits}}{
					Define $(\tilde{b}^{i,j})_{j=1}^k$ by $\tilde{b}^{i,j}_l:=\begin{cases} \tilde{b}^{i-1,j}_l &, l\neq i\\ \beta^j &, l=i\end{cases}$\;
					$\mathrm{AllWits}((i,(P^i_J)_{J\subseteq [k]}, (C^i_j)_{j=1}^k))\gets \mathrm{AllWits}((i,(P^i_J)_{ J\subseteq [k]}, (C^i_j)_{j=1}^k))\cup\{(\tilde{b}^{i,j})_{j=1}^k\}$\;}\color{black}
	}}}
	\textbf{return} $\mathcal{T}$\;
	\label{AlgDPValidCandsRounded}
	\caption{Dynamic program to compute rounded candidates}
\end{algorithm}
Before diving into the analysis of Algorithm~\ref{AlgDPValidCandsRounded}, we would like to provide some intuition about what is happening. We start by enumerating all possible values $(b^j_1)_{j=1}^k$ may attain in a solution $B$ with $b^k=a^*$ and use this information to compute $\mathrm{Cand}^1(B)$ according to Definition~\ref{DefCandForCover}. (Recall that this is independent of the values $b^j_l$ for $l\geq 2$, $j=1,\dots,k$.) Then, we round all non-zero entries of $\mathrm{Cand}^1(B)$ (except for the first one, which is $1$) down to the next power of $1+\delta$, $\delta=\frac{\epsilon}{4n}$. Each rounded candidate $\mathcal{C}$ is added to our table $\mathcal{T}$, and for back-tracing purposes, we store a cover $B$ that leads to $\mathcal{C}$ as $\mathrm{Witness}(\mathcal{C})$. For the analysis, we further maintain an imaginary map $\mathrm{AllWits}$ mapping each rounded candidate $\mathcal{C}\in\mathcal{T}$ to the set of all possible witness covers that result in $\mathcal{C}$ after (iterative) rounding.

After dealing with the base case $i=1$, we enumerate possible values of $(b^j_i)_{j=1}^k$ for $i=2,\dots,n$, loop over all rounded candidates for $i-1$ and compute new rounded candidates for $i$ according to Lemma~\ref{LemInductionStep}. Moreover, we deduce witnesses for our new candidates for $i$ from those stored for the candidates for $i-1$ and the values $(b^j_i)_{j=1}^k$. This might of course lead to an exponential growth of the size of the imaginary map $\mathrm{AllWits}$. However, the fact that the $\mathrm{Witness}$-map only memorizes one witness per candidate keeps the total running time under control, provided we can come up with a polynomial bound on the number of candidates we generate. Lemma~\ref{LemTPolynomial} takes care of this, and is the main ingredient of the proof of Theorem~\ref{TheoPolyTime}, which guarantees a polynomial running time.

\begin{restatable}{lemma}{LemTPolynomial} \label{LemTPolynomial}
At each point during the algorithm, we have $|\mathcal{T}| \le \alpha^{2^k} \cdot \beta^k \cdot n$, where
	\begin{align*}
		\alpha & = n + 2 -n \cdot \min \left\{\log_{1+\delta}(p^i_l) : i \in [n],\, l \in \{0,\dots,M+1\},\, p^i_l > 0\right\}, \\
		\beta & = n + 2 + \log_{1+\delta}(c_1 + \dots + c_n)-\log_{1+ \delta}(a_1)- \min \{ \log_{1+\delta} (c_i) : i \in [n]\}.
	\end{align*}
	In particular, for constant $k$, $|\mathcal{T}|$ is polynomially bounded in the encoding length of the given instance of the discrete Pareto cover problem and $\epsilon^{-1}$.
\end{restatable}
\begin{restatable}{theorem}{TheoPolyTime} \label{TheoPolyTime}
Given an instance $I$ of the discrete Pareto cover problem and a parameter $\epsilon\in(0,1)\cap\mathbb{Q}$ as input, Algorithm~\ref{AlgDPValidCandsRounded} runs in time polynomial in $\sz(I)$ and $\epsilon^{-1}$.
\end{restatable}

Denote the set of all candidate in $\mathcal{T}$ starting with $i$ by $\mathcal{T}_i$. In order to finally obtain an FPTAS for the discrete Pareto cover problem, our goal for the remainder of this section is to prove the following result:
\begin{restatable}{theorem}{TheoApproxGuarantee} \label{TheoApproxGuarantee}
Running Algorithm~\ref{AlgDPValidCandsRounded} and choosing the witness of a candidate of minimum cost in $\mathcal{T}_n$ yields a $(1+\epsilon)$-approximation.
\end{restatable}
The proof of Theorem~\ref{TheoApproxGuarantee} consists of two main steps. Lemma~\ref{LemCloseCosts} shows that for any rounded candidate $\tilde{\mathcal{C}}$ we store in $\mathcal{T}$ invokes similar costs to those of any of its witness covers. Proposition~\ref{PropAllSolutionsOccur} ensures that every cover $B=(b^j)_{j=1}^k$ with $b^k=a^*$ occurs as a possible witness for some candidate. In particular, this holds for an optimum cover $B^*$ (with $b^{*,k}=a^*$) and by Lemma~\ref{LemCloseCosts}, we can therefore infer that the costs of the solution we return can only be by a factor of $1+\epsilon$ larger than the optimum.
\begin{restatable}{lemma}{LemCloseCosts} \label{LemCloseCosts}
Let $\tilde{\mathcal{C}}\in\mathcal{T}_n$ and let $B\in\mathrm{AllWits}(\tilde{\mathcal{C}})$.
	Then $\mathrm{Cost}(\tilde{\mathcal{C}})\leq \ev_\mu[c_B]\leq (1+\epsilon)\cdot\mathrm{Cost}(\tilde{\mathcal{C}}).$
\end{restatable}
\begin{restatable}{proposition}{PropAllSolutionsOccur} \label{PropAllSolutionsOccur}
Let an instance of the discrete Pareto cover problem be given.
	For each $i\in[n]$ and each cover $B=(b^j)_{j=1}^k$ such that $b^k=a^*$ and $b^j_l=0$ for $j=1\dots,k-1$ and $l=i+1,\dots,n$, there exists $\mathcal{C}\in\mathcal{T}_i$ such that $B\in\mathrm{AllWits}(\mathcal{C})$.
\end{restatable}
Combining Theorem~\ref{TheoApproxGuarantee}, Lemma~\ref{LemTPolynomial} and Theorem~\ref{TheoPolyTime} and observing that we can compute the cost of a candidate in polynomial time for constant $k$, we obtain the following result:
\begin{theorem}
	There is an FPTAS for the discrete Pareto cover problem.
\end{theorem}
Together with Corollary~\ref{CorReductionFPTAS}, this finally yields Theorem~\ref{TheoFPTASGeneral}, which we restate once again.
\begin{RestateTheoFPTAS}
	Let $k \in \mathbb{N}$ be fixed.
Given a nice product measure $\mu$ on $[0,1]^n$ and $c \in \Q^n_{\ge 0}$, the problem of computing an (approximately) optimal Pareto cover of size $k$ with respect to $\mu$ and $c$ admits an FPTAS.
\end{RestateTheoFPTAS}

\section{Conclusion}
\label{secQuestions}
In this paper, we have introduced the Pareto cover problem and studied the case of product measures and linear cost functions. For fixed $k$, we have come to a pretty good understanding of its complexity: On the one hand, we could show weak NP-hardness of the problem. On the other hand, we have established the existence of an FPTAS.

However, there are several questions that remain open and constitute an interesting subject for future research. To begin with, we have seen that even in a very restricted setting such as the uniform probability distribution on $[0,1]^n$, it seems non-trivial to find an optimum cover (see Figure~\ref{figFirstExample}). Consequently, in order to obtain a better feeling for the problem at hand, it can be worthwhile to examine the structure of optimum solutions for such special cases.

  When dealing with the binary problem variant, another question that comes up is for which subsets of $\{0,1\}^n$, there exists an instance they are optimum for. Any insights towards this question may lead to new, perhaps more efficient strategies to tackle the Pareto cover problem.
 
  In addition to these rather concrete questions, there are also several more fundamental issues one may want to address. For instance, even though we have seen that computing the objective value attained by an arbitrary solution is \#P-hard in general (i.e., for non-fixed $k$), this does not resolve containment in NP since there might be another choice of a certificate that does the trick. More generally, it would be interesting to fully understand the dependence of the problem complexity on the parameter $k$. To this end, note that the complexity does not simply ``increase'' with larger values of $k$, given that in the discrete setting, the problem is weakly NP-hard for constant $k\geq 2$, and strongly NP-hard, e.g., for $k=\frac{n+5}{3}$, but once $k$ is at least as large as the number of all possible discrete vectors $b$, it is obvious what an optimum solution should look like (and we can output it in polynomial time, assuming an appropriate encoding of $k$ is chosen). Hence, it seems interesting to further investigate the hardness transition of the problem: When exactly does the problem become strongly NP-hard? When does it become easier again?
 
Finally, as all of our results apply to the case of product measures, it appears natural to ask what can be done for general probability measures. To this end, as alluded to in Section~\ref{secRelatedWork}, it could be fruitful to explore connections to existing results from statistical learning theory.

\bibliographystyle{plainurl}
\bibliography{references}

\begin{thebibliography}{1}

\bibitem{ausiello1999complexity}
Giorgio Ausiello, Pierluigi Crescenzi, Giorgio Gambosi, Viggo Kann, Alberto
  Marchetti-Spaccamela, and Marco Protasi.
\newblock {\em Complexity and Approximation: Combinatorial Optimization
  Problems and Their Approximability Properties}.
\newblock Springer Science \& Business Media, 1999.
\newblock \href {https://doi.org/10.1007/978-3-642-58412-1}
  {\path{doi:10.1007/978-3-642-58412-1}}.

\bibitem{garey1990computers}
Michael~R Garey and David~S Johnson.
\newblock {\em Computers and Intractability; A Guide to the Theory of
  NP-Completeness}.
\newblock W. H. Freeman \& Co., USA, 1990.
\newblock URL: \url{https://dl.acm.org/doi/10.5555/574848}.

\bibitem{korf1995approximate}
Richard~E Korf.
\newblock From approximate to optimal solutions: A case study of number
  partitioning.
\newblock In {\em Proceedings of the 14th joint conference on Artificial
  intelligence}, pages 266--272, 1995.
\newblock URL: \url{https://dl.acm.org/doi/abs/10.5555/1625855.1625890}.

\bibitem{luedtke2010integer}
James Luedtke, Shabbir Ahmed, and George~L Nemhauser.
\newblock An integer programming approach for linear programs with
  probabilistic constraints.
\newblock {\em Mathematical programming}, 122(2):247--272, 2010.
\newblock \href {https://doi.org/10.1007/s10107-008-0247-4}
  {\path{doi:10.1007/s10107-008-0247-4}}.

\bibitem{nemirovski2007convex}
Arkadi Nemirovski and Alexander Shapiro.
\newblock Convex approximations of chance constrained programs.
\newblock {\em SIAM Journal on Optimization}, 17(4):969--996, 2007.
\newblock \href {https://doi.org/10.1137/050622328}
  {\path{doi:10.1137/050622328}}.

\bibitem{vadhan2001complexity}
Salil~P Vadhan.
\newblock The complexity of counting in sparse, regular, and planar graphs.
\newblock {\em SIAM Journal on Computing}, 31(2):398--427, 2001.
\newblock \href {https://doi.org/10.1137/S0097539797321602}
  {\path{doi:10.1137/S0097539797321602}}.

\end{thebibliography}

\appendix

\section{Two useful inequalities}
In the subsequent arguments, we employ the two well-known inequalities 
\begin{equation}
\forall x\in\mathbb{R}: 1+x\leq \eul^x \label{EqOneplusX}
\end{equation}
and
\begin{equation}
\forall x\in [0,1]: \eul^x\leq 1+(\eul-1)\cdot x\leq 1+2x\label{EqBoundExpX},
\end{equation}
which follow by convexity of the exponential function. More precisely, \eqref{EqOneplusX} arises by looking at the tangent in $(0,1)$, while studying the secant passing through $(0,1)$ and $1,\eul)$ gives rise to \eqref{EqBoundExpX}.

\section{Omitted details from the hardness proof for $k=2$\label{SecAppendixHardnessk=2}}
In order to establish Lemma~\ref{lemHardnessProbabilities}, we need to explain how to obtain close approximations of the values of the exponential function that occur in our reduction in polynomial time. Lemma~\ref{LemExpApprox} takes care of this.
\begin{lemma}
	Let $x\in [0,2]\cap\mathbb{Q}$ and $m\in\mathbb{N}$. Then we can compute a rational number $y$ with $(1-2^{-m})\cdot \eul^{-x}\leq y \leq (1+2^{-m})\cdot \eul^{-x}$ in time polynomial in $m$ and the encoding length of $x$.\label{LemExpApprox}
	\end{lemma}
\begin{proof}
As $x\in [0,2]$, we have $\eul^{-x}\in [\eul^{-2}, 1]\subseteq \left[\frac{1}{8},1\right]$. Hence, it suffices to approximate $\eul^{-x}$ up to an additive error of $2^{-(m+3)}$. We know that we have
\[\eul^{-x}=\sum_{n=0}^\infty \frac{(-x)^n}{n!},\] where $0!:=1$.	Define $N:=\max\{32, 2m+4\}$. Note that $N$ is even. We can compute \[\sum_{n=0}^N \frac{(-x)^n}{n!}\] in time polynomial in $m$ and the encoding length of $x$, so if we can further show that
\[\left|\sum_{n=N+1}^\infty \frac{(-x)^n}{n!}\right|\leq 2^{-(m+3)},\] we are done. To this end, we compute
\begin{align*}
\left|\sum_{n=N+1}^\infty \frac{(-x)^n}{n!}\right|&\leq \sum_{n=N+1}^\infty \left|\frac{(-x)^n}{n!}\right|=\sum_{n=N+1}^\infty \frac{x^n}{n!}\\
&\leq \sum_{n=N+1}^\infty \frac{x^n}{\frac{N}{2}\cdot \left(\frac{N}{2}+1\right)\cdot\dots \cdot n}\leq \sum_{n=N+1}^\infty \frac{x^n}{\left(\frac{N}{2}\right)^{n-\frac{N}{2}}}\\
&\leq \sum_{n=N+1}^\infty \frac{x^n}{\left(\frac{N}{2}\right)^{\frac{n}{2}}}=\sum_{n=N+1}^\infty\left(\frac{\sqrt{2}\cdot x}{\sqrt{N}}\right)^n\\
&\stackrel{x\in[0,2], N\geq 32}{\leq}\sum_{n=N+1}^\infty\left(\frac{\sqrt{2}\cdot 2}{\sqrt{32}}\right)^n=\sum_{n=N+1}^\infty\left(\frac{1}{2}\right)^n\\&=2\cdot \left(\frac{1}{2}\right)^{N+1}\leq \left(\frac{1}{2}\right)^{2m+4}\leq 2^{-(m+3)}.
\end{align*}
\end{proof}
\lemHardnessProbabilities*
\begin{proof}
First, we show how to compute $(p_i)_{i=1}^n$ such that 	\[
\frac{1-\beta}{\eul^{\alpha \cdot a_i}} \le 1-p_i \le \frac{1+\beta}{\eul^{\alpha \cdot a_i}}
\quad
\text{for all } i \in [n].\]
As \[\alpha\cdot a_i=2\cdot\frac{a_i}{\sum_{j=1}^n a_j}\in [0,2]\quad
\text{for all } i \in [n],\] Lemma~\ref{LemExpApprox} tells us that for each $i$, we can compute a rational number $q_i$ with \[(1-\beta)\cdot\eul^{-\alpha\cdot a_i} \leq q_i\leq (1+\beta)\cdot \eul^{-\alpha\cdot a_i}\] in time polynomial in \[\log(\beta^{-1})\in\mathcal{O}\left(\log\left(\sum_{i=1}^n a_i\right) + \log(n)\right)\] and the encoding length of \[\alpha\cdot a_i=\frac{2}{\sum_{l=1}^n a_l}\cdot a_i.\]  In particular, we can compute such numbers $(q_i)_{i=1}^n$ in time polynomial in the enconding length $\sz\left((a_i)_{i=1}^n\right)$ of the \texttt{PARTITION} instance $a_1,\dots,a_n$. Setting $p_i:=1-q_i$  for $i=1,\dots, n$ meets all requirements.

 Finally, we discuss how to determine a rational number $\gamma$ with \[\frac{(1-\beta)^{n+2}}{\alpha \cdot \eul}
\le \sum_{i=1}^n a_i - \gamma
\le \frac{(1-\beta)^n}{\alpha \cdot \eul}.
\] By Lemma~\ref{LemExpApprox}, we can first compute $\gamma_0$ with $(1-\beta)\cdot \eul^{-1}\leq \gamma_0\leq (1+\beta)\cdot \eul^{-1}$ in time polynomial in $\log(\beta^{-1})$, which we have already verified to be polynomial in $\sz\left((a_i)_{i=1}^n\right)$. Then, we set $\gamma_1:=(1-\beta)^{n+1}\cdot\alpha^{-1}\cdot \gamma_0$. As the encoding lengths of $\alpha$ and $\beta$ are polynomial in $\sz\left((a_i)_{i=1}^n\right)$, the latter can also be done in polynomial time. $\gamma_1$ now satisfies
\[(1-\beta)^{n+2}\cdot\alpha^{-1}\cdot\eul^{-1}\leq \gamma_1\leq (1-\beta)^{n+1}\cdot (1+\beta)\cdot\alpha^{-1}\cdot\eul^{-1} < (1-\beta)^n\cdot\alpha^{-1}\cdot \eul^{-1}.\]
Finally, assigning $\gamma:=\sum_{i=1}^n a_i -\gamma_1$ does the trick.	
\end{proof}
Now, we want to prove Lemma~\ref{lemHardnessReduction}. To this end, it suffices to verify the following two claims:
\begin{claim}
	If there is $I\subseteq[n]$ such that $\sum_{i\in I} a_i = \frac{1}{2}\cdot\sum_{i=1}^n a_i$, then $\ev_{\mu_p}[c_B]\leq \gamma$, where $B=\{b,\onevec\}$ and $b\in\{0,1\}^n$ is given by $b_i = 0 \Leftrightarrow i\in I$. \label{ClaimPartToBundle}
\end{claim}
\begin{claim}
	If there is no $I\subseteq[n]$ such that $\sum_{i\in I} a_i = \frac{1}{2}\cdot\sum_{i=1}^n a_i$, then any cover $B$ attains costs of $\ev_{\mu_p}[c_B]> \gamma$.\label{ClaimBundleToPart}
\end{claim}
Recall that by positivity of the $p_i$ and the $c_i$, any optimum cover $B$ is of the form $B=\{b,\onevec\}$ with $b\in\{0,1\}^n$, and each such cover is clearly feasible. We start by obtaining upper and lower bounds on the objective values attained by such covers.
\begin{proposition}
Let $B=\{b,\onevec\}$ with $b\in\{0,1\}^n$ and let $I:=\{i\in [n]:b_i=0\}$. Then 
\[\sum_{i=1}^n a_i-(1+\beta)^n\cdot h\left(\sum_{i\in I} a_i\right)\leq \ev_{\mu_p}[c_B]\leq\sum_{i=1}^n a_i-(1-\beta)^n\cdot h\left(\sum_{i\in I} a_i\right),\]
where $h:\R\rightarrow\R, \mapsto x\cdot\eul^{-\alpha\cdot x}$.\label{PropUpperLowerBoundObj}
\end{proposition}
\begin{proof}
	We have
\begin{align}
\ev_{\mu_p}[c_B]&=c\t b\cdot\mu_p(\{x\in \{0,1\}^n: x\leq b\})+c\t\onevec\cdot\mu_p(\{x\in\{0,1\}^n: x\not\le b\}) \notag\\
&=c\t b\cdot\mu_p(\{x\in \{0,1\}^n: x\leq b\})+c\t\onevec\cdot(1-\mu_p(\{x\in\{0,1\}^n: x\le b\}))\notag\\
&=c\t\onevec - (c\t\onevec - c\t b)\cdot\mu_p(\{x\in \{0,1\}^n: x\leq b\})\notag\\
&=\sum_{i=1}^n a_i-\left(\sum_{i=1}^n a_i -\sum_{i\in [n]\setminus I} a_i\right)\cdot\prod_{i\in I}(1-p_i)\notag\\
&=\sum_{i=1}^n a_i-\sum_{i\in I} a_i\cdot \prod_{i\in I}(1-p_i)\label{EqObjectiveExp}.
\end{align}
By our choice of the $p_i$, we obtain
\[(1-\beta)^n\cdot \eul^{-\alpha\cdot \sum_{i\in I} a_i}\leq \prod_{i\in I} (1-\beta)\cdot \eul^{-\alpha\cdot a_i}\leq \prod_{i\in I}(1-p_i)\leq \prod_{i\in I} (1+\beta)\cdot \eul^{-\alpha\cdot a_i}\leq (1+\beta)^n\cdot\eul^{-\alpha\cdot\sum_{i\in I} a_i}.\]
Plugging this into \eqref{EqObjectiveExp} results in
\[\sum_{i=1}^n a_i-(1+\beta)^n\cdot h\left(\sum_{i\in I} a_i\right)\leq \ev_{\mu_p}[c_B]\leq \sum_{i=1}^n a_i-(1-\beta)^n\cdot h\left(\sum_{i\in I} a_i\right)\]	as claimed.
\end{proof}
Now, we are ready to prove Claim~\ref{ClaimPartToBundle}.
\begin{proof}[Proof of Claim~\ref{ClaimPartToBundle}]
	Pick $I\subseteq[n]$ such that $\sum_{i\in I} a_i = \frac{1}{2}\cdot\sum_{i=1}^n a_i=\alpha^{-1}$, let \mbox{$b\in\{0,1\}^n$} with  $b_i = 0 \Leftrightarrow i\in I$, and define $B:=\{b,\onevec\}$. By Proposition~\ref{PropUpperLowerBoundObj}, we obtain
	\begin{align*}
	\ev_{\mu_p}[c_B]&\leq\sum_{i=1}^n a_i-(1-\beta)^n\cdot h\left(\sum_{i\in I} a_i\right)=\sum_{i=1}^n a_i-(1-\beta)^n\cdot h(\alpha^{-1})\\
	&=\sum_{i=1}^n a_i-(1-\beta)^n\cdot \alpha^{-1}\cdot\eul^{-1}\leq \gamma.
	\end{align*}
\end{proof}
Next, we would like to prove Claim~\ref{ClaimBundleToPart}. Pick an optimum cover $B$ is of the form $B=\{b,\onevec\}$ with $b\in\{0,1\}^n$ and let $I:=\{i\in[n]:b_i=0\}$. Now, we make the following simple, but yet very useful observation. As $\sum_{i\in I} a_i \neq \frac{1}{2}\cdot\sum_{i=1}^n a_i=\alpha^{-1}$ and both the left-hand and the right-hand side are integers, we must either have $\sum_{i\in I} a_i \leq \alpha^{-1}-1$, or $\sum_{i\in I} a_i\geq \alpha^{-1}+1$. Now, Lemma~\ref{LemAwayFromOpt} comes into play, which provides an upper bound on $h\left(\sum_{i\in I} a_i\right)$, which, by Proposition~\ref{PropUpperLowerBoundObj} results in a lower bound on $\ev_{\mu_p}[c_B]$.
\begin{lemma}
Let $\alpha\in(0,1]$ and let $h:\mathbb{R}\rightarrow\mathbb{R}, x\mapsto x\cdot\eul^{-\alpha\cdot x}$. If $|x_0-\alpha|\geq 1$, then $h(x_0)\leq h(\alpha^{-1})\cdot (1-0.25\cdot\alpha^2\cdot \eul^{-1})$.\label{LemAwayFromOpt}
\end{lemma}
\begin{proof}
The first derivative of $h$ is given by $h':\mathbb{R}\rightarrow\mathbb{R}, x\mapsto (1-\alpha\cdot x)\cdot\eul^{-\alpha\cdot x}.$ Hence, for $x\leq \alpha^{-1}$, $h'(x)\geq 0$, and for $x\geq \alpha^{-1}$, $h'(x)\leq 0$. Thus, for $x_0\leq \alpha^{-1}-1$, we obtain
\begin{align*}
h(x_0)&=h(\alpha^{-1})-\int_{x_0}^{\alpha^{-1}}h'(x)dx \leq h(\alpha^{-1})-\int_{\alpha^{-1}-1}^{\alpha^{-1}-0.5}h'(x)dx\\
&\leq h(\alpha^{-1})-0.5\cdot \min_{x\in [\alpha^{-1}-1,\alpha^{-1}-0.5]} h'(x) \\&= h(\alpha^{-1})-0.5\cdot \min_{x\in [\alpha^{-1}-1,\alpha^{-1}-0.5]} (1-\alpha\cdot x)\cdot\eul^{-\alpha\cdot x}\\
		&\leq h(\alpha^{-1})-0.5\cdot 0.5\cdot\alpha \cdot \eul^{\alpha-1}\\
		&=\alpha^{-1}\cdot \eul^{-1}-0.25\cdot \alpha\cdot\eul^{\alpha}\cdot\eul^{-1}\\
		&=\alpha^{-1}\cdot \eul^{-1}\cdot (1-0.25\cdot\alpha^2\cdot \eul^{\alpha})=h(\alpha^{-1})\cdot (1-0.25\cdot\alpha^2\cdot \eul^{\alpha})\\
		&\stackrel{\alpha\geq 0}{\leq} h(\alpha^{-1})\cdot (1-0.25\cdot\alpha^2\cdot \eul^{-1}).
\end{align*}
Similarly, for $x_0\geq \alpha^{-1}+1$, we calculate
\begin{align*}
h(x_0)&=h(\alpha^{-1})+\int_{\alpha^{-1}}^{x_0}h'(x)dx \leq h(\alpha^{-1})+\int_{\alpha^{-1}+0.5}^{\alpha^{-1}+1}h'(x)dx\\
&\leq h(\alpha^{-1})+0.5\cdot \max_{x\in [\alpha^{-1}+0.5,\alpha^{-1}+1]} h'(x) \\&= h(\alpha^{-1})+0.5\cdot \max_{x\in [\alpha^{-1}+0.5,\alpha^{-1}+1]} (1-\alpha\cdot x)\cdot\eul^{-\alpha\cdot x}\\
&\leq h(\alpha^{-1})+0.5\cdot (-0.5\cdot\alpha) \cdot \eul^{-(\alpha+1)}\\
&=\alpha^{-1}\cdot \eul^{-1}-0.25\cdot \alpha\cdot\eul^{-\alpha}\cdot\eul^{-1}\\
&=\alpha^{-1}\cdot \eul^{-1}\cdot (1-0.25\cdot\alpha^2\cdot \eul^{-\alpha})\stackrel{\alpha\leq 1}{\leq}h(\alpha^{-1})\cdot (1-0.25\cdot\alpha^2\cdot \eul^{-1}).
\end{align*}
\end{proof}
Now, we can complete the proof of Claim~\ref{ClaimBundleToPart}.
\begin{proof}[Proof of Claim~\ref{ClaimBundleToPart}]
Pick an optimum cover $B$ that is of the form $B=\{b,\onevec\}$ with $b\in\{0,1\}^n$ and let $I:=\{i\in[n]:b_i=0\}$. By the previous arguments, Proposition~\ref{PropUpperLowerBoundObj} and Lemma~\ref{LemAwayFromOpt}, we obtain
\begin{align*}
\ev_{\mu_p}[c_B]&\geq \sum_{i=1}^n a_i-(1+\beta)^n\cdot h\left(\sum_{i\in I} a_i\right)\geq \sum_{i=1}^n a_i-(1+\beta)^n\cdot(1-0.25\cdot\alpha^2\cdot \eul^{-1})\cdot h(\alpha^{-1})\\
&= \sum_{i=1}^n a_i-(1+\beta)^n\cdot(1-0.25\cdot\alpha^2\cdot \eul^{-1})\cdot \alpha^{-1}\cdot\eul^{-1}.
\end{align*}
By definition, we have \[\gamma\leq \sum_{i=1}^n a_i-(1-\beta)^{n+2}\cdot\alpha^{-1}\cdot\eul^{-1},\] so it suffices to prove that \[(1+\beta)^n\cdot (1-0.25\cdot\alpha^2\cdot \eul^{-1}) < (1-\beta)^{n+2}.\]
To this end, we calculate
\begin{align*}
(1-\beta)^{n+2}\cdot (1+\beta)^{-n}&=(1-\beta)^{n+2}\cdot \left(1-\frac{\beta}{1+\beta}\right)^n\geq (1-\beta)^{2n+2}\\&\stackrel{(*)}{\geq} \eul^{-2\cdot\beta\cdot (2n+2)}\stackrel{\eqref{EqOneplusX}}{\geq} 1-2\cdot\beta\cdot (2n+2)=1-4\cdot(n+1)\cdot \frac{\alpha^2}{48\cdot(n+1)}\\&=1-\frac{\alpha^2}{12}>1-0.25\cdot\alpha^2\cdot \eul^{-1},
\end{align*}
which completes the proof. Note that the inequality marked ($*$) follows from the fact that by convexity of the exponential function, for $\lambda\in[0,1]$, we have
\[1-\frac{\lambda}{2}\geq 1-\lambda\cdot(1-\eul^{-1})=\lambda\cdot\eul^{-1}+(1-\lambda)\cdot\eul^0\geq \eul^{\lambda\cdot(-1)+(1-\lambda)\cdot 0}=\eul^{-\lambda}.\] Plugging in $\lambda:=2\beta$ yields $1-\beta\geq \eul^{-2\cdot\beta}$, implying ($*$).
\end{proof}
\section{NP-hardness for $k\geq 3$ \label{SecHardnessKgeq3}}
\subsection{The case $k=3$}\label{SecHardnessKequal3}
Next, we deal with the case $k=3$. As before, we perform a reduction from \texttt{PARTITION}.
Before presenting the formal hardness proof, we first give some intuition on the underlying idea. We choose costs and probabilities proportional to the $a_i$, but scale down all probabilities by dividing by a very large constant $L$, that is, we set $c_i:=a_i$ and $p_i:=\frac{a_i}{L}$ for $i=1,\dots,n$. The motivation behind this is that if we choose, say, $L > \left(\sum_{i=1}^n a_i\right)^2$, then we enforce that $\zerovec$ is contained in any optimum solution $B$. This is because if we do not include $\zerovec$, by integrality of the $a_i$, with probability one, we have to select $b\neq\zerovec$ of cost $c\t b\geq 1$, leading to an objective value of at least $1$. On the other hand, even if we always have to pay for $\onevec$ whenever $x\neq\zerovec$, the total costs are still bounded by
\[\sum_{i=1}^n p_i\cdot \sum_{i=1}^n a_i = L^{-1}\cdot \left(\sum_{i=1}^n a_i\right)^2 < 1.\] As all $p_i$ are positive, including $\onevec$ into $B$ is mandatory, too. Hence, any optimum solution $B$ is of the form $\{\zerovec, b, \onevec\}$, where $b\in\{0,1\}^n$ (as all costs are strictly positive). Pick such a solution $B=\{\zerovec,b,\onevec\}$ and let $I:=\{i\in [n]: b_i = 1\}$. Then
\begin{align*}\ev_{\mu_p}[c_B]&=\mu_p(\{\zerovec\})\cdot c\t 0 + \mu_p(\{x\in\{0,1\}^n: \zerovec\neq x\leq b\})\cdot c\t b+\mu_p(\{x\in\{0,1\}^n:x\not\leq b\})\cdot c\t\onevec\\
&=-\mu_p(\{x\in\{0,1\}^n: \zerovec\neq x\leq b\})\cdot (c\t\onevec-c\t b)+\underbrace{(1-\mu_p(\{\zerovec\})\cdot c\t\onevec}_{\text{constant}}.\end{align*}
The latter term is constant, so minimizing the objective is equivalent to maximizing the product $\mu_p(\{x\in\{0,1\}^n: \zerovec\neq x\leq b\})\cdot (c\t\onevec-c\t b)$. We further observe that for $L$ small enough,
$\mu_p(\{x\in\{0,1\}^n: \zerovec\neq x\leq b\})\approx \sum_{i\in I} p_i=\frac{1}{L}\cdot\sum_{i\in I} a_i$ since the probability that at least two coordinates assume the value $1$ is by an order of $L^{-1}$ smaller than the probability that this happens for one individual coordinate. This yields \[\mu_p(\{x\in\{0,1\}^n: \zerovec\neq x\leq b\})\cdot (c\t\onevec-c\t b) \approx \frac{1}{L}\cdot\sum_{i\in I}a_i\cdot\left(\sum_{i=1}^n a_i -\sum_{i\in I} a_i\right).\] The right-hand side attains the maximum possible value of $\frac{1}{L}\cdot \left(\frac{\sum_{i=1}^n a_i}{2}\right)^2$ if and only if there is a feasible solution to the \texttt{PARTITION}-problem. Otherwise, by integrality of the $a_i$, the right-hand side is by at least $\frac{1}{L}$ smaller. Hence, choosing the threshold $\gamma$ for the decision variant of the Pareto cover problem appropriately allows us to determine whether $(a_i)_{i=1}^n$ is a ``yes'' instance. The remainder of this section contains a formal hardness proof. In order to conduct it, we need the following auxiliary statement:
\begin{proposition}
	Let $(\alpha_i)_{i=1}^n\in [0,1]^n$. Then $\sum_{i=1}^n \alpha_i\geq 1-\prod_{i=1}^n (1-\alpha_i)$.\label{PropProductBound}
	\end{proposition}
\begin{proof}
Consider the product probability measure $\mu$ given $\mu(\{x\})=\prod_{i:x_1=1}\alpha_i\cdot\prod_{i:x_i=0}(1-\alpha_i)$ for $x\in\{0,1\}^n$. The left-hand side is $\sum_{i=1}^n \mu(\{x\in\{0,1\}^n: x_i=1\})$, the right-hand side is $\mu(\bigcup_{i=1}^n\{x\in\{0,1\}^n: x_i=1\}) $.
\end{proof}
\begin{theorem}
	The decision variant of the Pareto cover problem with $k=3$ is (weakly) NP-hard.
\end{theorem}
\begin{proof}
	Given an instance $(a_i)_{i=1}^n$ of \texttt{PARTITION}, we transform it into the following instance of the decision variant of the Pareto cover problem with $k=3$:
	We set \[M:=2\cdot\left(\sum_{j=1}^n a_j\right)^2, \]\[p_i:=\frac{a_i}{\sum_{i=1}^n a_i\cdot M}, i=1,\dots,n\]
	and \[c_i:=a_i,i=1,\dots,n.\]
	Moreover, we set \[\gamma:=\left(1-\prod_{i=1}^n (1-p_i)\right)\cdot\sum_{i=1}^n a_i +\frac{\sum_{i=1}^n a_i}{M^2}-\frac{1}{\sum_{i=1}^n a_i\cdot M}\cdot \frac{(\sum_{i=1}^n a_i)^2}{4}.\]
	The encoding length of each of these numbers is polynomial in the input size. Moreover, we have
	\begin{align*}\gamma &< \left(1-\prod_{i=1}^n (1-p_i)\right)\cdot\sum_{i=1}^n a_i +\frac{\sum_{i=1}^n a_i}{M^2}\stackrel{Prop.~\ref{PropProductBound}}{\leq}\sum_{i=1}^n p_i\cdot\sum_{i=1}^n a_i +\frac{\sum_{i=1}^n a_i}{M^2}\\&=\frac{\sum_{i=1}^n a_i}{\sum_{i=1}^n a_i\cdot M}\cdot\sum_{i=1}^n a_i +\frac{\sum_{i=1}^n a_i}{M^2}=\frac{\sum_{i=1}^n a_i}{M}+\frac{\sum_{i=1}^n a_i}{M^2}\leq \frac{1}{2}+\frac{1}{2^2}<1.\end{align*} Hence, as before, we can deduce that any cover of cost at most $\gamma < 1$ has to comprise both $\zerovec$ and $\onevec$.
	To see that our reduction is correct, it therefore suffices to prove the following claims:
	\begin{claim}
	Let $I\subseteq[n]$ with $\sum_{i\in I} a_i=\frac{1}{2}\cdot\sum_{i=1}^n a_i$ and define $b\in\{0,1\}^n$ by $b_i=1\Leftrightarrow i\in I$. Then $B:=\{\zerovec,b,\onevec\}$ is a cover with $\ev_{\mu_p}[c_B]\leq \gamma$.\label{ClaimPartitionToBundle}
	\end{claim}
	\begin{claim}
		If there is no $I\subseteq[n]$ such that $\sum_{i\in I} a_i = \frac{1}{2}\cdot\sum_{i=1}^n a_i$, then any cover $B$ attains costs of $\ev_{\mu_p}[c_B]> \gamma$.	\label{ClaimBundleToPartition}
	\end{claim}
\begin{proof}[Proof of Claim~\ref{ClaimPartitionToBundle}]
	Let $I$ such that $\sum_{i\in I} a_i=\frac{1}{2}\cdot\sum_{i=1}^n a_i$ and define $b$ and $B$ as in the statement of the claim. Denote the $i$-th unit vector by $\xi^i$, that is $\xi^i_i=1$ and $\xi^i_j=0$ for $j\neq i$.  We have \begingroup\allowdisplaybreaks\begin{align*}
	\ev_{\mu_p}[c_B]&=-\mu_p(\{x\in\{0,1\}^n: \zerovec\neq x\leq b\})\cdot (c\t\onevec-c\t b)+(1-\mu_p(\{\zerovec\})\cdot c\t\onevec\\
	\ev_{\mu_p}[c_B]&\leq-\mu_p(\{\xi^i,i\in I\})\cdot (c\t\onevec-c\t b)+(1-\mu_p(\{\zerovec\})\cdot c\t\onevec\\
	&= -\sum_{i\in I} p_i\cdot \prod_{j\in[n]\setminus \{i\}}(1-p_j)\cdot (c\t\onevec-c\t b)+(1-\mu_p(\{\zerovec\})\cdot c\t\onevec\\
	&\stackrel{Prop.~\ref{PropProductBound}}{\leq} -\sum_{i\in I} p_i\cdot \left(1-\sum_{j=1}^n p_j\right)\cdot (c\t\onevec-c\t b)+(1-\mu_p(\{\zerovec\})\cdot c\t\onevec\\
	&\leq (1-\mu_p(\{\zerovec\}))\cdot c\t\onevec+\left(\sum_{i=1}^n p_i\right)^2\cdot (c\t\onevec-c\t b)-\sum_{i\in I} p_i\cdot (c\t\onevec-c\t b)\\
	&\leq\left(1-\prod_{i=1}^n (1-p_i)\right)\cdot\sum_{i=1}^n a_i +\left(\frac{\sum_{i=1}^n a_i}{\sum_{i=1}^n a_i\cdot M}\right)^2\cdot\sum_{i=1}^n a_i-\frac{\sum_{i\in I}a_i\cdot \sum_{i\in[n]\setminus I} a_i}{\sum_{i=1}^n a_i\cdot M}\\
	&=\left(1-\prod_{i=1}^n (1-p_i)\right)\cdot\sum_{i=1}^n a_i +\frac{\sum_{i=1}^n a_i}{M^2}-\frac{1}{\sum_{i=1}^n a_i\cdot M}\cdot \frac{(\sum_{i=1}^n a_i)^2}{4}=\gamma.
	\end{align*}
	\endgroup
	This finishes the proof.
\end{proof}
\begin{proof}[Proof of Claim~\ref{ClaimBundleToPartition}]
	Assume towards a contradiction that there were a cover $B$ with $\ev_{\mu_p}[c_B]\leq \gamma$.
Then $B=\{\zerovec,b,\onevec\}$ with $b\in\{0,1\}^n$. Let $I:=\{i\in [n]: b_i=1\}$.
	Then \[\sum_{i\in I} a_i\neq \frac{1}{2}\cdot\sum_{i=1}^{n}a_i,\] and by integrality of the $a_i$ and $\frac{1}{2}\cdot\sum_{i=1}^{n}a_i$, this implies 
	\[\sum_{i\in I} a_i\cdot\sum_{i\in [n]\setminus I} a_i\leq \frac{1}{4}\cdot \left(\sum_{i=1}^n a_i\right)^2-1.\]
	We obtain
	\begingroup \allowdisplaybreaks
	\begin{align*}
	\ev_{\mu_p}[c_B]&=-\mu_p(\{x\in\{0,1\}^n: \zerovec\neq x\leq b\})\cdot (c\t\onevec-c\t b)+(1-\mu_p(\{\zerovec\})\cdot c\t\onevec\\
	&\geq -\sum_{i\in I} p_i\cdot (c\t\onevec-c\t b)+(1-\mu_p(\{\zerovec\})\cdot c\t\onevec\\
	&=\left(1-\prod_{i=1}^n (1-p_i)\right)\cdot\sum_{i=1}^n a_i-\frac{1}{\sum_{i=1}^n a_i\cdot M}\cdot \sum_{i\in I} a_i \cdot \sum_{i\in [n]\setminus I} a_i\\
	&\geq \left(1-\prod_{i=1}^n (1-p_i)\right)\cdot\sum_{i=1}^n a_i-\frac{1}{\sum_{i=1}^n a_i\cdot M}\cdot\left(\frac{1}{4}\cdot \left(\sum_{i=1}^n a_i\right)^2-1\right)\\
	&=\left(1-\prod_{i=1}^n (1-p_i)\right)\cdot\sum_{i=1}^n a_i +\frac{1}{\sum_{i=1}^n a_i\cdot M}-\frac{1}{\sum_{i=1}^n a_i\cdot M}\cdot \frac{(\sum_{i=1}^n a_i)^2}{4}\\
	&>\left(1-\prod_{i=1}^n (1-p_i)\right)\cdot\sum_{i=1}^n a_i +\frac{\sum_{i=1}^n a_i}{M^2}-\frac{1}{\sum_{i=1}^n a_i\cdot M}\cdot \frac{(\sum_{i=1}^n a_i)^2}{4}=\gamma.
	\end{align*}
	\endgroup
\end{proof}
\end{proof}
\subsection{The case $k\geq 4$}\label{SecHardnesskgeq4}
In this section, we investigate the hardness of the decision variant of the binary Pareto cover problem for $k\geq 4$. To this end, we consider the \parti problem \cite{korf1995approximate}, which is defined as follows: The input consists of positive integers $t$ and $a_1,\dots,a_m$ such that $\sum_{i=1}^m a_i$ is divisible by $t$. The task is to decide whether there exists a partition of $[m]$ into $t$ subsets $I_1,\dots,I_t$ such that $\sum_{i\in I_j} a_i =\frac{1}{t}\cdot\sum_{i=1}^m a_i$ for all $j=1,\dots,t$. For fixed constant values of $t$, this problem is usually referred to as the \mwp{t} problem (see \cite{korf1995approximate}), which is known to be weakly NP-complete for $t\geq 2$. To this end, note that the \mwp{2} problem is just the \texttt{PARTITION} problem defined above. Moreover, for $t\geq 2$, there is a simple reduction from \texttt{PARTITION} to \mwp{t} by simply mapping $(a_i)_{i=1}^m$ to $(a'_i)_{i=1}^{m+t-2}$, where $a'_i=a_i$ for $i=1,\dots,m$ and $a'_i=\frac{1}{2}\cdot\sum_{i=1}^m a_i$ for $i=m+1,\dots,m+t-2$. For $t=\frac{m}{3}$, we obtain the \texttt{3-PARTITION} problem, which is strongly NP-complete (see \cite{garey1990computers}). Our main result in this section is the following theorem:
\begin{theorem}
	There is a polynomial time reduction from the \parti problem to the decision variant of the binary Pareto cover problem. It maps a \parti instance with parameters $t$ and $m$ to an instance of the Pareto cover problem with $n=m+1$ coordinates and cover size $k=t+2$.\label{TheoReductionkgeq4}
	\end{theorem}
\begin{corollary}
For any fixed constant $k\geq 4$, the decision variant of the binary Pareto cover problem restricted to instances with cover size $k$ is weakly NP-complete.	
\end{corollary}
\begin{proof}
Membership in NP has been shown in Section~\ref{secDiscreteInNP}. For weak NP-hardness, observe that Theorem~\ref{TheoReductionkgeq4} provides a polynomial time reduction from the weakly NP-complete \mwp{k-2} problem to the binary Pareto cover problem restricted to instances with cover size $k$.	
\end{proof}
\begin{corollary}
	The binary Pareto cover problem, restricted to instances where $k=\frac{n+5}{3}$, is strongly NP-hard.
	\end{corollary}
\begin{proof}
The \texttt{3-PARTITION} problem corresponds to the restriction of the \parti problem to instances where $t=\frac{m}{3}$. Theorem~\ref{TheoReductionkgeq4} hence provides a reduction from the strongly NP-complete \texttt{3-PARTITION} to the Pareto cover problem, restricted to instances where $k=t+2=\frac{m}{3}+2=\frac{n-1}{3}+2=\frac{n+5}{3}$.
\end{proof}
The reduction from \mwp{t} to the binary Pareto cover problem with $k=t+2$ and $n=m+1$ is very similar to the one presented in the previous section: For each $i=1,\dots,m$, we set $c_i:=a_i$ and choose $p_i:=\frac{a_i}{L}$, where $L$ is a very large constant. However, in contrast to what we did for $k=3$, we add an $m+1$-st coordinate, for which we set $p_{m+1}$ to a tiny value, e.g., $L^{-2}$ whilst $c_{m+1}$ is chosen to be extremely large.

 This construction is motivated as follows: First of all, by choosing all probabilities to be tiny, $\mu_p(\{\zerovec\})$ becomes very close to $1$ and by essentially the same argument as in Section~\ref{SecHardnessKequal3}, this ensures that any cover $B$ for which $\ev_{\mu_p}[c_B]$ is smaller than a certain threshold must contain $\zerovec$. But for any cover $B$ with $\zerovec\in B$, $\ev_{\mu_p}[c_B]$ is dominated by the costs of covering the first $m$ unit vectors $\xi^i$, $i=1,\dots,m$. This is because for $L$ large enough, $\mu_p(\{x\in \{0,1\}^{m+1}: \onevec\t x\geq 2 \vee x_{m+1}=1\})$ is by an order of magnitude smaller than each of the $\mu_p(\{\xi^i\})$. While including $\onevec$ into the cover is mandatory to ensure feasibility as all probabilities are positive, for any $i=1,\dots,m$, no optimum solution will contain more than one $b\neq \onevec$ such that $b_i=1$: If there were more than one, we would only need the cheaper one to cover $\xi^i$ and can set $b_i=0$ for the other one, decreasing the costs. On the other hand, if we choose $c_{n+1}$ sufficiently large, then we are  forced to cover each of the $\xi^i$ by some $b$ with $b_{m+1}\neq 1$ (and hence, $b_{m+1}=0$) to stay below a cost threshold of, say, $1$ again. But this means that if we have a cover $B$ with $\ev_{\mu_p}[c_B]<1$, then $(I_b:=\{i:b_i=1\})_{b\in B\setminus\{\zerovec,\onevec\}}$ yields a $t$-partition of $[m]$. Each $ b\in B\setminus\{\zerovec,\onevec\}$ covers all $\xi^i$ with $i\in I_b$. Hence, we obtain
 \[\ev_{\mu_p}[c_B]\approx \sum_{i=1}^m \mu_p(\{\xi^i\})\cdot\min_{b\in B: \xi^i\leq b}c\t b = \sum_{b\in B\setminus\{\zerovec,\onevec\}}\sum_{i\in I_b} p_i\cdot \sum_{i\in I_b}a_i=\frac{1}{L}\cdot\sum_{b\in B\setminus\{\zerovec,\onevec\}}\left(\sum_{i\in I_b}a_i\right)^2.\] By convexity of the function $x\mapsto x^2$, the rightmost term attains its minimum possible value if and only if $(I_b)_{b\in B\setminus\{\zerovec,\onevec\}}$ defines a solution to \mwp{t}, that is, if the \mwp{t} instance is a ``yes'' instance.
 
The remainder of this section contains a formal proof of Theorem~\ref{TheoReductionkgeq4}.
	Let an instance $t$, $(a_i)_{i=1}^m$ of \parti be given and let $M:=13\cdot\left(\sum_{i=1}^m a_i\right)^2.$ Consider the instance of the decision variant of the binary Pareto cover problem given by $k=t+2$, $n=m+1$, $p_i:=\frac{a_i}{M^4}$ and $c_i:=a_i$, for $i=1,\dots,m$, $p_{m+1}:=\frac{1}{M^6}$, $c_{m+1}:= 2M$ and $\gamma:=\frac{1}{M^4}\cdot\frac{\left(\sum_{i=1}^m a_i\right)^2}{t}+\frac{6}{M^5}$.
Note that the encoding lengths of $M$, and, hence, $(p_i)_{i=1}^{m+1}$, $(c_i)_{i=1}^{m+1}$ and $\gamma$ are polynomial in the input size. Therefore, it remains to prove the following two claims:
\begin{claim}
Let $(I_j)_{j=1}^{t}$ be a partition of $[m]$ such that $\sum_{i\in I_j}a_i =\frac{1}{t}\cdot\sum_{i=1}^m a_i$ for all $j=1,\dots,t$. For $j=1,\dots,t$, let $b^j\in\{0,1\}^{m+1}$ be given by $b^j_i=1 \Leftrightarrow i\in I_j$ and let $B:=\{\zerovec,\onevec\}\cup\{b^j,j=1,\dots,t\}$. Then $\ev_{\mu_p}[c_B]\leq \gamma.$\label{Claimkm1PartitionToBundles}
	\end{claim}
\begin{claim}
If $B$ is a cover such that $\ev_{\mu_p}[c_B]\leq \gamma$, then $(a_i)_{i=1}^m$ is a ``yes'' instance of the \parti problem.\label{ClaimBundlesTokm1Partition}
\end{claim}
\begin{proof}[Proof of Claim~\ref{Claimkm1PartitionToBundles}]
	We have
	\begin{align*}
	\ev_{\mu_p}[c_B]&\leq\mu_p(\{\zerovec\})\cdot\min_{b\in B} c\t b+\sum_{i=1}^m \mu_p(\{\xi^i\})\cdot\min_{b\in B: \xi^i\leq b}c\t b+\left(1-\mu_p(\{\zerovec\})-\sum_{i=1}^m\mu_p(\{\xi^i\})\right)\cdot c\t\onevec\\
	&=\mu_p(\{\zerovec\})\cdot 0+\sum_{j=1}^{t}\sum_{i\in I_j}\mu_p(\{\xi^i\})\cdot\min_{b\in B: \xi^i\leq b}c\t b+\left(1-\mu_p(\{\zerovec\})-\sum_{i=1}^m\mu_p(\{\xi^i\})\right)\cdot c\t\onevec\\
	&\leq \sum_{j=1}^{t}\sum_{i\in I_j} p_i\cdot c\t b^j+\left(1-\prod_{i=1}^{m+1}(1-p_i)-\sum_{i=1}^m p_i\cdot\prod_{l\in[m+1]\setminus\{i\}} (1-p_l)\right)\cdot c\t\onevec\\
	&\stackrel{Prop.~\ref{PropProductBound}}{\leq}\sum_{j=1}^{t}\sum_{i\in I_j} p_i\cdot \sum_{i\in I_j} a_i+\left(\sum_{i=1}^{m+1} p_i-\sum_{i=1}^m p_i\cdot\prod_{l\in[m+1]} (1-p_l)\right)\cdot c\t\onevec\\
	&=\sum_{j=1}^{t}\sum_{i\in I_j} p_i\cdot \sum_{i\in I_j} a_i+\left(p_{m+1}+\sum_{i=1}^{m} p_i\cdot\left(1-\prod_{l\in[m+1]} (1-p_l)\right)\right)\cdot\sum_{i=1}^{m+1}a_i\\
	&\stackrel{Prop.~\ref{PropProductBound}}{\leq}\sum_{j=1}^{t}\sum_{i\in I_j} p_i\cdot \sum_{i\in I_j} a_i+\left(p_{m+1}+\left(\sum_{i=1}^{m+1}p_{i}\right)^2\right)\cdot\sum_{i=1}^{m+1}a_i\\
	&\leq\frac{1}{M^4}\cdot\sum_{j=1}^{t}\left(\sum_{i\in I_j}a_i\right)^2+\left(\frac{1}{M^6}+\left(\frac{1}{M^4}\cdot\sum_{i=1}^m a_i+\frac{1}{M^6}\right)^2\right)\cdot \left(\sum_{i=1}^{m}a_i+2M\right)\\
	&\leq \frac{1}{M^4}\cdot\sum_{j=1}^{t}\left(\frac{\sum_{i=1}^m a_i}{t}\right)^2+\left(\frac{1}{M^6}+\left(\frac{1}{2\cdot M^3}+\frac{1}{M^6}\right)^2\right)\cdot 3M\\
	&\leq \frac{1}{M^4}\cdot\sum_{j=1}^{t}\left(\frac{\sum_{i=1}^m a_i}{t}\right)^2+\left(\frac{1}{M^6}+\left(\frac{1}{M^3}\right)^2\right)\cdot 3M\\
		&\leq \frac{1}{M^4}\cdot\frac{\left(\sum_{i=1}^m a_i\right)^2}{t}+\frac{2}{M^6}\cdot 3M=\frac{1}{M^4}\cdot\frac{\left(\sum_{i=1}^n a_i\right)^2}{t}+\frac{6}{M^5}=\gamma.
	\end{align*}
\end{proof}
Before we dive into the proof of Claim~\ref{ClaimBundlesTokm1Partition}, we prove two inequalities that we will need later.
\begin{equation}
\prod_{i=1}^{m+1}(1-p_i)\stackrel{Prop.~\ref{PropProductBound}}{\geq} 1-\sum_{i=1}^{m}p_i-p_{m+1}=1-\frac{\sum_{i=1}^m a_i}{M^4}-\frac{1}{M^6}\geq 1-2\cdot\frac{1}{2\cdot M^3}=1-\frac{1}{M^3}>\frac{1}{2}\label{EqBoundPropZero}
\end{equation}
\begin{equation}
\frac{1}{2}>\frac{1}{M^3}\geq\frac{1}{2\cdot\left(\sum_{i=1}^m a_i\right)^2}\cdot \frac{1}{M^3}\cdot\frac{\left(\sum_{i=1}^m a_i\right)^2}{t}+\frac{1}{12}\cdot\frac{6}{M^3}\geq \frac{1}{M^4}\cdot\frac{\left(\sum_{i=1}^m a_i\right)^2}{t}+\frac{6}{M^5}=\gamma\label{EqTrivialBoundGamma}
\end{equation}
\begin{proof}[Proof of Claim~\ref{ClaimBundlesTokm1Partition}]
	Let $B$ be a cover such that $\ev_{\mu_p}[c_B]\leq\gamma$. We can w.l.o.g. assume that $B\subseteq\{0,1\}^n$. As all $p_i$ are strictly positive, feasibility implies that $\onevec\in B$. We want to show that also $\zerovec\in B$. Assume towards a contradiction that $\zerovec\not\in B$. As all $c_i$ are positive integers, this implies that $c\t b\geq 1$ for all $b\in B$, leading to \[\ev_{\mu_p}[c_B]\geq\mu_p(\{0,1\}^{m+1})\cdot\min_{b\in B}c\t b\geq 1\stackrel{\eqref{EqTrivialBoundGamma}}{>} \gamma,\] a contradiction.\\
	Next, we prove that for every $i\in[m]$, there is $b\in B\setminus\{\onevec,\zerovec\}$ such that $\xi^i\leq b$, i.e., $b_i=1$. Again, assume towards a contradiction that this was not the case. Then there is $i\in[m]$ such that the only element of $B$ covering $\xi^i$ is $\onevec$. Consequently,
	\[\ev_{\mu_p}[c_B]\geq\mu_p(\{\xi^i\})\cdot c\t\onevec = p_i\cdot\prod_{l\in[m+1]\setminus\{i\}}(1-p_l)\cdot\left(\sum_{i=1}^m a_i + 2\cdot M\right)\stackrel{\eqref{EqBoundPropZero}}{\geq} \frac{a_i}{M^4}\cdot\frac{1}{2}\cdot 2\cdot M\geq \frac{1}{M^3}\stackrel{\eqref{EqTrivialBoundGamma}}{>}\gamma,\] contradicting our assumptions. Hence, for each $i\in[m]$, there is $b\in B\setminus\{\onevec,\zerovec\}$ with $b_i=1$. As a consequence, we can pick a map $\pi:[m]\rightarrow B\setminus\{\zerovec,\onevec\}$ mapping $i$ to an element of $\mathrm{argmin}\{c\t b: b\in B, \xi^i\leq b\}$. Not that as all $c_i$ are strictly positive, $\onevec$ will never attain the minimum. Now, by definition, the sets $\pi^{-1}(b)$, $b\in B\setminus\{\zerovec,\onevec\}$ form a partition of $[m]$ into $t$ parts. We claim that this partition certifies that $(a_i)_{i=1}^m$ is a ``yes'' instance, meaning that we have $\sum_{i\in\pi^{-1}(b)} a_i=\frac{1}{t}\cdot\sum_{i=1}^m a_i$ for all $b\in B\setminus\{\zerovec,\onevec\}$. To this end, we observe that
	\begin{align*}
	\gamma &\geq \ev_{\mu_p}[c_B]\geq \sum_{i=1}^m\mu_p(\{\xi^i\})\cdot \min_{b\in B: \xi^i\leq b} c\t b\\
	&\geq \sum_{i=1}^m p_i\cdot\prod_{l=1}^{m+1}(1-p_l)\cdot c\t\pi(i)\\
	&= \prod_{l=1}^{m+1}(1-p_l)\cdot\sum_{b\in B\setminus\{\zerovec,\onevec\}} \sum_{i\in\pi^{-1}(b)}p_i\cdot c\t b\\
	&\stackrel{i\in \pi^{-1}(b)\Rightarrow b_i=1}{\geq} \prod_{l=1}^{m+1}(1-p_l)\cdot\sum_{b\in B\setminus\{\zerovec,\onevec\}} \sum_{i\in\pi^{-1}(b)}p_i\cdot \sum_{i\in\pi^{-1}(b)}a_i\\
	&=\frac{1}{M^4}\cdot \prod_{l=1}^{m+1}(1-p_l)\cdot\sum_{b\in B\setminus\{\zerovec,\onevec\}}\left(\sum_{i\in\pi^{-1}(b)}a_i\right)^2\\
	&\stackrel{\eqref{EqBoundPropZero}}{\geq}\left(1-\frac{1}{M^3}\right)\cdot \frac{1}{M^4}\cdot\sum_{b\in B\setminus\{\zerovec,\onevec\}}\left(\sum_{i\in\pi^{-1}(b)}a_i\right)^2.
	\end{align*}
	By the inequality between arithmetic and quadratic mean, we have
	\[\sum_{b\in B\setminus\{\zerovec,\onevec\}}\left(\sum_{i\in\pi^{-1}(b)}a_i\right)^2\geq t\cdot\left(\frac{\sum_{b\in B\setminus\{\zerovec,\onevec\}}\sum_{i\in\pi^{-1}(b)}a_i}{t}\right)^2=\frac{\left(\sum_{i=1}^m a_i\right)^2}{t},\]
	and this is tight if and only if $\sum_{i\in\pi^{-1}(b)}a_i=\frac{1}{t}\cdot\sum_{i=1}^m a_i$ for all $b\in B\setminus\{\zerovec,\onevec\}$. If this holds, we are done. Otherwise, by integrality of the left and the right-hand side (recall that $\sum_{i=1}^m a_i$ is divisible by $t$), we can infer that
	\[\sum_{b\in B\setminus\{\zerovec,\onevec\}}\left(\sum_{i\in\pi^{-1}(b)}a_i\right)^2\geq \frac{\left(\sum_{i=1}^m a_i\right)^2}{t}+1.\]
	This results in \begin{align*}\gamma &\geq\ev_{\mu_p}[c_B]\geq\left(1-\frac{1}{M^3}\right)\cdot \frac{1}{M^4}\cdot \left(\frac{\left(\sum_{i=1}^m a_i\right)^2}{t}+1\right)\\&=\frac{1}{M^4}\cdot\frac{\left(\sum_{i=1}^m a_i\right)^2}{t}+\frac{1}{M^4}-\frac{1}{M^7}\cdot\frac{\left(\sum_{i=1}^m a_i\right)^2}{t}-\frac{1}{M^7}\\
	&\geq \frac{1}{M^4}\cdot\frac{\left(\sum_{i=1}^m a_i\right)^2}{t}+\frac{1}{M^4}-2\cdot\frac{1}{2\cdot M^6}\\
	&\geq  \frac{1}{M^4}\cdot\frac{\left(\sum_{i=1}^m a_i\right)^2}{t}+\frac{1}{2\cdot M^4}>\frac{1}{M^4}\cdot\frac{\left(\sum_{i=1}^m a_i\right)^2}{t}+\frac{6}{M^5}=\gamma, \end{align*}
	a contradiction. This completes the proof.
\end{proof}
We conclude this section by remarking that $k=\frac{n+5}{3}$ is not the only cover size for which Theorem~\ref{TheoReductionkgeq4} proves NP-hardness. In fact, by performing reductions from the \texttt{3-PARTITION} problem, one can show that for any integer $p\geq 3$, the problem of deciding whether, given positive integers $a_1,\dots,a_m$ such that $m$ is divisible by $p$ and $\sum_{i=1}^m a_i$ is divisible by $\frac{m}{p}$, there exists a partition of $[m]$ into $\frac{m}{p}$ sets $(I_j)_{j=1}^{\frac{m}{p}}$ such that $\sum_{i\in I_j} a_i = \frac{p}{m}\cdot\sum_{i=1}^m a_i$ for all $j=1,\dots,\frac{m}{p}$, is NP-complete. This then shows NP-hardness of the Pareto cover problem restricted to instances where $k=\frac{n-1}{p}+2$. Finally, it seems likely that by adding additional numbers of value $\frac{p+1}{m}\cdot\sum_{i=1}^m a_i$ to the instance, one can cover some cases ``between'' $k=\frac{n-1}{p+1}+2$ and $k=\frac{n-1}{p}+2$. However, as this is not the main focus of our paper, we omit the details here.
\section{Proofs omitted from Section~\ref{SecGeneralToDiscrete}}
\label{secAppendixGeneralToDiscrete}

\PropExpectationBoundsOpt*
\begin{proof}
Consider the random variable $\bar{c}:[0,1]^n\rightarrow\mathbb{R}_{\geq 0},x\mapsto \sum_{i=1}^n c_i\cdot x_i$. Then $\bar{c}(x)\leq c_B(x)$ for all $x\in [0,1]^n$, implying that \[\ev_{\prod_{i=1}^n\mu_i}[c_B]\geq \ev_{\prod_{i=1}^n\mu_i}[\bar{c}]=\sum_{i=1}^n c_i\cdot\ev_{X_i\sim\mu_i}[X_i]\geq \alpha\cdot \sum_{i=1}^n c_i.\]
\end{proof}
\PropSizeMeps*
\begin{proof}
We have \begin{align*}M_\epsilon &= \lceil \log_{1+\epsilon} ((\epsilon\cdot\alpha)^{-1})\rceil\leq \log_{1+\epsilon}((\epsilon\cdot\alpha^{-1})+1 =\frac{\log(\epsilon^{-1})+\log(\alpha^{-1})}{\log(1+\epsilon)}+1\\&\stackrel{(*)}{\leq} 2\cdot\frac{\log(\epsilon^{-1})+\log(\alpha^{-1})}{\epsilon}+1\leq2\cdot\frac{1+\log(\alpha^{-1})}{\epsilon^2}+1  \in\mathcal{O}\left(\frac{\log(\alpha^{-1})}{\epsilon^2}\right),\end{align*}
	where the inequality marked ($*$) follows from the fact that \[\log(1+\epsilon)\geq \ln(1+\epsilon)=\ln(1)+\int_{1}^{1+\epsilon} \frac{1}{x}dx \geq \ln(1)+\frac{\epsilon}{2}=\frac{\epsilon}{2}.\]
\end{proof}
\LemGoodDiscretizedSolution*
\begin{proof}
Pick an optimum solution $B=(b^j)_{j=1}^k$ and define the $\epsilon$-discrete solution $\hat{B}=(\hat{b}^j)_{j=1}^k$ by rounding each value $b_i^j$ to the next largest element of $Q(\epsilon)$. Note that this cannot destroy the property that the set of uncovered elements in $[0,1]^n$ has measure $0$. We have \[\ev_{\prod_{i=1}^n\mu_i}(c_{\hat{B}})=\ev_{\prod_{i=1}^n\mu_i}(c_B)+\ev_{\prod_{i=1}^n\mu_i}(c_{\hat{B}}-c_B).\] By definition, we have $b^j\leq \hat{b}^j$ for each $j=1,\dots,k$, implying that for $x$ with $c_B(x)=c(b^j)$, we have \begin{align*}c_{\hat{B}}(x)&\leq c(\hat{b^j}) = \sum_{i=1}^n c_i\cdot \hat{b}^j_i\leq \sum_{i=1}^n c_i\cdot \max\{\epsilon\cdot\alpha,(1+\epsilon)\cdot b^j_i\}\leq (1+\epsilon)\cdot\sum_{i=1}^n c_i\cdot b^j_i+\alpha\cdot\epsilon\cdot\sum_{i=1}^n c_i\\&\stackrel{Prop.~\ref{PropExpectationBoundsOpt}}{\leq}(1+\epsilon)\cdot c(b^j)+\epsilon\cdot\mathrm{OPT}=(1+\epsilon)\cdot c_B(x)+\epsilon\cdot\mathrm{OPT}.\end{align*}
	This yields $c_{\hat{B}}-c_B\leq \epsilon\cdot c_B + \epsilon\cdot \mathrm{OPT}$ and, hence,
	\begin{align*}\ev_{\prod_{i=1}^n\mu_i}[c_{\hat{B}}]&=\ev_{\prod_{i=1}^n\mu_i}[c_B]+\ev_{\prod_{i=1}^n\mu_i}[c_{\hat{B}}-c_B]\leq \ev_{\prod_{i=1}^n\mu_i}[c_B]+\ev_{\prod_{i=1}^n\mu_i}[\epsilon\cdot c_B + \epsilon\cdot \mathrm{OPT}]\\&=(1+\epsilon)\cdot\ev_{\prod_{i=1}^n\mu_i}[c_B]+\epsilon\cdot\mathrm{OPT}=(1+2\epsilon)\cdot\mathrm{OPT}\leq \left(1+\frac{\gamma}{15}\right)\cdot\mathrm{OPT}. \end{align*}
\end{proof} 
\LemSimilarObjective*
\begin{proof}
First, note that $(\tilde{\mu}_i)_{i=1}^n$ is well-defined in that we do not divide by $0$ and each $\tilde{\mu}_i$ constitutes a discrete probability distribution on $Q(\epsilon)$. The first part follows from the fact that \[\sum_{l=0}^{M_\epsilon+1} \sigma_i(q_{l-1},q_l,\epsilon)\geq (1+\epsilon)^{-1}\cdot\sum_{l=0}^{M_\epsilon+1} \mu_i((q_{l-1},q_l]\cap [0,1])=(1+\epsilon)^{-1}\cdot\mu_i([0,1])=(1+\epsilon)^{-1},\] the second part from the normalization of the probabilities.\\ Also using the upper bound we have on the values $\sigma_i(q_{l-1},q_l,\epsilon)$, we obtain \[(1+\epsilon)^{-1}\leq \sum_{l=0}^{M_\epsilon+1} \sigma_i(q_{l-1},q_l,\epsilon)\leq (1+\epsilon).\]
	This implies
	\begin{align}
	(1+\epsilon)^{-2}\cdot \mu_i((q_{l-1},q_l]\cap[0,1])&\leq \frac{\sigma_i(q_{l-1},q_l,\epsilon)}{\sum_{t=0}^{M_\epsilon +1}\sigma_i(q_{t-1},q_t,\epsilon)}= \tilde{\mu}_i(\{q_l\})\notag\\&=\tilde{\mu}_i((q_{l-1},q_l]\cap[0,1])\notag\\&\leq (1+\epsilon)^2\cdot \mu_i((q_{l-1},q_l]\cap[0,1])\label{EqMeasuresClose}
	\end{align}
	Now, pick some $\epsilon$-discrete cover $B=(b^j)_{j=1}^k$. We first want to show that feasibility of $B$ w.r.t.\ $(\tilde{\mu}_i)_{i=1}^n$ implies feasibility w.r.t.\ $(\mu_i)_{i=1}^n$. We know that $B$ is feasible with respect to the discrete measures $(\tilde{\mu}_i)_{i=1}^n$ if and only if there is $j\in[k]$ such that $b^j\geq (q_{l_i})_{i=1}^n$, where $l_i:=\max\{l: p^i_l>0\}$. In particular, the set $U$ of uncovered points $x\in [0,1]^n$ is contained within $\bigcup_{i=1}^n \bigtimes_{t=1}^{i-1}[0,1]\times(q_{l_i},1]\times \bigtimes_{t=i+1}^n [0,1]$, and we know that $\tilde{\mu}_i((q_{l_i},1])=0$ for each $i$. But by \eqref{EqMeasuresClose}, this implies $\mu_i((q_{l_i},1])=0$ for all $i$. Hence, by definition of the product measure, $(\prod_{i=1}^n \mu_i)(U)=0$, so $B$ is feasible for $(\mu_i)_{i=1}^n$. For the other direction, note that any $\epsilon$-discrete cover for $(\mu_i)_{i=1}^n$ must contain $b$ with $b\geq (q_{l'_i})_{i=1}^n$, where $l'_i$ is maximum such that $\mu_i((q_{l'_i-1},q_{l'_i}])>0$. But by \eqref{EqMeasuresClose}, the values $l'_i$ and $l_i$ coincide for $i\in[n]$. This shows feasibility for $(\tilde{\mu}_i)_{i=1}^n$.
	
	Now, we want to show that
	\[\left(1+\frac{\gamma}{15}\right)^{-1}\cdot\ev_{\prod_{i=1}^n\mu_i}[c_B]\leq \ev_{\prod_{i=1}^n\tilde{\mu}_i}[c_B]\leq \left(1+\frac{\gamma}{15}\right)\cdot\ev_{\prod_{i=1}^n\mu_i}[c_B].\]
	In order to compute $\ev_{\prod_{i=1}^n\mu_i}[c_B]$ and $\ev_{\prod_{i=1}^n\tilde{\mu}_i}[c_B]$, our strategy is to partition $[0,1]^n$ into a finite family of (measurable) subsets on which $c_B$ is constant.
	To this end, let $\mathcal{I}:=\{0,\dots,M_\epsilon+1\}^n$, and, for $I=(l_i)_{i=1}^n\in\mathcal{I}$, define \[S_I:=\bigtimes_{i=1}^n (q_{l_i-1}, q_{l_i}]\cap [0,1].\] As $q_{-1}=-1$ and $q_{M_\epsilon+1}=1$, $[0,1]^n =\dot{\bigcup}_{I\in \mathcal{I}} S_I$ clearly defines a partition of $[0,1]^n$. Moreover, for $x\in[0,1]^n$ and $j\in[k]$, we have \[x\leq b^j\Leftrightarrow\forall i: x_i\leq b^j_i\stackrel{b^j\in Q(\epsilon)^n}{\Leftrightarrow} \forall i: \min\{q\in Q(\epsilon):q\geq x_i\}\leq b^j_i.\] But for any $x\in S_I$ with $I=(l_i)_{i=1}^n$, $\min\{q\in Q(\epsilon):q\geq x_i\}=q_{l_i}$, so the set of all $b^j$ that cover $x$ is the same for all $x\in S_I$ and in particular, the costs of a cheapest such vector agree. Hence, $c_B$ is constant on $S_I$. Denote the value $c_B$ assumes on $S_I$ by $c_I$. Then
	\begin{align*}
	(1+\epsilon)^{-2n}\cdot \ev_{\prod_{i=1}^n\mu_i}[c_B]&= (1+\epsilon)^{-2n}\cdot\sum_{I=(l_i)_{i=1}^n\in\mathcal{I}} \prod_{i=1}^n\mu_i((q_{l_i-1},q_{l_i}]\cap[0,1])\cdot c_I\\
	&= \sum_{I=(l_i)_{i=1}^n\in\mathcal{I}} \prod_{i=1}^n(1+\epsilon)^{-2}\cdot\mu_i((q_{l_i-1},q_{l_i}]\cap[0,1])\cdot c_I\\
	&\stackrel{\eqref{EqMeasuresClose}}{\leq} \underbrace{\sum_{I=(l_i)_{i=1}^n\in\mathcal{I}} \prod_{i=1}^n\tilde{\mu}_i((q_{l_i-1},q_{l_i}]\cap[0,1])\cdot c_I}_{=\ev_{\prod_{i=1}^n\tilde{\mu}_i}[c_B]}\\
	&\stackrel{\eqref{EqMeasuresClose}}{\leq}\sum_{I=(l_i)_{i=1}^n\in\mathcal{I}} \prod_{i=1}^n(1+\epsilon)^{2}\cdot\mu_i((q_{l_i-1},q_{l_i}]\cap[0,1])\cdot c_I\\
	&=(1+\epsilon)^{2n}\cdot \ev_{\prod_{i=1}^n\mu_i}[c_B].
	\end{align*}
	By our choice of $\epsilon=\frac{\gamma}{60n}$, we obtain
	\[(1+\epsilon)^{2n}\stackrel{\eqref{EqOneplusX}}{\leq} \eul^{\epsilon\cdot 2n}=\eul^{\frac{\gamma}{30}}\stackrel{\eqref{EqBoundExpX}}{\leq} 1+2\cdot\frac{\gamma}{30}=1+\frac{\gamma}{15}.\]
\end{proof}
\TheoReduction*
\begin{proof}
Let an instance $I:=(n,k,(\sigma_i)_{i=1}^n, (c_i)_{i=1}^n, \alpha)$ of the Pareto cover problem be given and let $\gamma\in\mathbb{Q}\cap(0,1)$. Define $\epsilon:=\frac{\gamma}{60n}$ and consider the instance $I_\gamma$ of the discrete Pareto cover problem given by $n$, $k$, $Q(\epsilon)$, $\left((p^i_l)_{l=0}^{M_\epsilon+1}\right)_{i=1}^n$ and $(c_i)_{i=1}^n$. As before, call the corresponding discrete probability distributions $(\tilde{\mu}_i)_{i=1}^n$. Denote the optimum for $I$ by $\mathrm{OPT}$ and the optimum for $I_\gamma$ by $\mathrm{OPT}_\gamma$. By Lemma~\ref{LemGoodDiscretizedSolution}, there exists an $\epsilon$-discrete cover $B$ with $\ev_{\prod_{i=1}^n\mu_i}[c_B]\leq \left(1+\frac{\gamma}{15}\right)\cdot\mathrm{OPT}$. Lemma~\ref{LemSimilarObjective}, therefore, yields
	\[\mathrm{OPT}_\gamma\leq \ev_{\prod_{i=1}^n\tilde{\mu}_i}[c_B]\leq \left(1+\frac{\gamma}{15}\right)\cdot \ev_{\prod_{i=1}^n\mu_i}[c_B]\leq \left(1+\frac{\gamma}{15}\right)^2\cdot\mathrm{OPT}.\] Let $B^*$ be a cover of cost at most  $\left(1+\frac{\gamma}{15}\right)\cdot\mathrm{OPT}_\gamma$ for $I_\gamma$. By Lemma~\ref{LemSimilarObjective}, $B^*$ is feasible for $I$. Then 
	\[\ev_{\prod_{i=1}^n\mu_i}[c_{B^*}]\stackrel{Lem.~\ref{LemSimilarObjective}}{\leq}\left(1+\frac{\gamma}{15}\right)\cdot\ev_{\prod_{i=1}^n\tilde{\mu}_i}[c_{B^*}]\leq \left(1+\frac{\gamma}{15}\right)^2\cdot\mathrm{OPT}_\gamma\leq \left(1+\frac{\gamma}{15}\right)^4\cdot\mathrm{OPT}.\]
	Finally, \[\left(1+\frac{\gamma}{15}\right)^4 = 1+4\cdot \frac{\gamma}{15}+6\cdot\left(\frac{\gamma}{15}\right)^2+4\cdot\left(\frac{\gamma}{15}\right)^3+\left(\frac{\gamma}{15}\right)^4\stackrel{\gamma\in(0,1)}{<}1+15\cdot\frac{\gamma}{15}=1+\gamma.\]
	This shows that we obtain the claimed approximation guarantee, and it remains to take care of the running time. To this end, Proposition~\ref{PropSizeMeps} tells us that $M_\epsilon\in\mathcal{O}(\log(\alpha^{-1})\cdot\epsilon^{-2})=\mathcal{O}(\log(\alpha^{-1})\cdot\gamma^{-2}\cdot n^2)$, which is polynomial in the instance size and $\gamma^{-1}$. The encoding length of $\epsilon$ is polynomial in the input size, and as each of the values in $Q(\epsilon)$ is either $0$ or $1$ or originates from $\epsilon$, $\alpha$ and $1$ by $\mathcal{O}(M_\epsilon)$ many additions and multiplications, their encoding lengths are polynomially bounded in the input size and $\gamma^{-1}$ as well. This in turn implies that all of our (polynomially many) oracle calls we need to compute the $p^i_l$ run in polynomial time. Hence, $I_\gamma$ can be derived from $I$ in time polynomial in $\gamma^{-1}$ and the encoding length of $I$.
\end{proof}
\CorReductionFPTAS*
\begin{proof}
	The fact that we obtain a $(1+\gamma)$-approximate solution for $I$ is a direct consequence of Theorem~\ref{TheoReduction}, so it remains to show that the total running time is polynomial in the encoding length of $I$ and $\gamma^{-1}$. But we have already seen that $I_\gamma$ can be compute within this time bound and (thus) has a encoding length that is polynomially bounded in the size of $I$ and $\gamma^{-1}$. As the FPTAS must have a polynomial running time in terms of the size of $I_\gamma$ and $\gamma^{-1}$, the claim follows.
\end{proof}

\section{Lemmata and proofs omitted from Section~\ref{SecRounding}}
\label{secAppendixRounding}
\begin{lemma}
	Let $\tilde{\mathcal{C}}_i:=(i,(\tilde{P}_J)_{J\subseteq [k]}, (\tilde{C}_j)_{j=1}^k)\in\mathcal{T}_i$, let $B=(b^j)_{j=1}^k\in\mathrm{AllWits}(\tilde{\mathcal{C}}_i)$ and let
	$\mathcal{C}_i:=(i,(P_J)_{J\subseteq [k]}, (C_j)_{j=1}^k):=\mathrm{Cand}^i(B)).$\\
	Then $\forall J\subseteq [k]: \tilde{P}_J\leq P_J\leq (1+\delta)^i\cdot \tilde{P}_J$ and
	$\forall j\in [k]: \tilde{C}_j\leq C_j\leq (1+\delta)^i\cdot \tilde{C}_j.$\label{LemRatios}
\end{lemma}
\begin{proof}
Induction on $i$. For $i=1$, we have, by Lemma~\ref{LemInductionStart},
	\[\tilde{P}_J=\begin{cases}
	(1+\delta)^{\lfloor\log_{1+\delta}P_J\rfloor} &, P_J >0\\
	0 &, P_J=0
	\end{cases}\]
	and \[\tilde{C}_j\gets \begin{cases}
	(1+\delta)^{\lfloor\log_{1+\delta}C_j\rfloor} &, C_j >0\\
	0 &, C_j=0
	\end{cases}\]
	Hence, $\tilde{P}_J\leq P_J\leq (1+\delta)\cdot \tilde{P}_J$ for all $J\subseteq[k]$ and
	$\tilde{C}_j\leq C_j\leq (1+\delta)\cdot \tilde{C}_j$ for all $j=1,\dots,k$ as desired.\\
	Now, let $i\geq 2$ and assume that the assertion holds for all smaller values of $i$. Pick $\tilde{\mathcal{C}}_i:=(i,(\tilde{P}^i_J)_{J\subseteq [k]}, (\tilde{C}^i_j)_{j=1}^k)\in\mathcal{T}$, $B^{i}=(b^{i,j})_{j=1}^k\in\mathrm{AllWits}(\tilde{\mathcal{C}}_i)$ and let 
	\[\mathcal{C}_i:=(i,(P^i_J)_{J\subseteq [k]}, (C^i_j)_{j=1}^k):=\mathrm{Cand}^i(B^{i}).\]
	Define $\tilde{\mathcal{C}}_{i-1}:=(i-1,(\tilde{P}^{i-1}_J)_{J\subseteq [k]}, (\tilde{C}^{i-1}_j)_{j=1}^k)\in\mathcal{T}$ to be the last candidate for which $B^i$ is added to $\mathrm{AllWits}(\tilde{\mathcal{C}}_i)$ in the respective iteration of the foreach-loop in line~\ref{LineForLoopAllWits}. The corresponding $B^{i-1}=(b^{i-1,j})_{j=1}^k\in\mathrm{AllWits}(\tilde{C}_{i-1})$ has to satisfy $b^{i-1,j}_l=
	b^{i,j}_l$ for $l\neq i$ and the corresponding vector $(\beta^j)_{j=1}^k$ must be given by $\beta^j=b^{i,j}_i$. Now, let \[\mathcal{C}_{i-1}:=(i-1,(P^{i-1}_J)_{ J\subseteq [k]}, (C^{i-1}_j)_{j=1}^k):=\mathrm{Cand}^{i-1}((B^{i-1,j})_{j=1}^k).\] Then Lemma~\ref{LemInductionStep} tells us that
	\[P^i_J=\sum_{\substack{L:\\ J\subseteq L\subseteq[k]}} P^{i-1}_L\cdot (\prod_{l=1}^n \mu_l)((\max_{j\in L\setminus J} \beta^j,\min_{j\in J} \beta^j]\cap[0,1]).\] Moreover,
	$C^i_j=C^{i-1}_j + c_i\cdot \beta^j$.\\
	Let $\hat{C}^i_j$ and $\hat{P}^i_J$ denote the values of $\tilde{C}^i_j$ and $\tilde{P}^i_J$ before the rounding to a power of $(1+\delta)$ is performed. Then 
	\[\hat{P}^i_J=\sum_{\substack{L:\\ J\subseteq L\subseteq[k]}} \tilde{P}^{i-1}_L\cdot (\prod_{l=1}^n \mu_l)((\max_{j\in L\setminus J} \beta^j,\min_{j\in J} \beta^j]\cap[0,1])\] and
	$\hat{C}^i_j=\tilde{C}^{i-1}_j + c_i\cdot \beta^j$.
	By the induction hypothesis, we obtain
	\begin{align*}\hat{P}^i_J&=\sum_{\substack{L:\\ J\subseteq L\subseteq[k]}} \tilde{P}^{i-1}_L\cdot (\prod_{l=1}^n \mu_l)((\max_{j\in L\setminus J} \beta^j,\min_{j\in J} \beta^j]\cap[0,1])\\
	&\leq \underbrace{\sum_{\substack{L:\\ J\subseteq L\subseteq[k]}} P^{i-1}_L\cdot (\prod_{l=1}^n \mu_l)((\max_{j\in L\setminus J} \beta^j,\min_{j\in J} \beta^j]\cap[0,1])}_{=P^i_J}\\
	&\leq \sum_{\substack{L:\\ J\subseteq L\subseteq[k]}} (1+\delta)^{i-1}\cdot\tilde{P}^{i-1}_L\cdot (\prod_{l=1}^n \mu_l)((\max_{j\in L\setminus J} \beta^j,\min_{j\in J} \beta^j]\cap[0,1])\\
	&\leq(1+\delta)^{i-1}\cdot \hat{P}^i_J,
	\end{align*}
	so \begin{equation}
	\hat{P}^i_J\leq P^i_J\leq (1+\delta)^{i-1}\cdot \hat{P}^i_J\label{EqPhat}.
	\end{equation}
	As \[\tilde{P}^i_J=\begin{cases} (1+\delta)^{\lfloor\log_{1+\delta}\hat{P}^i_J\rfloor} &, \hat{P}^i_J>0\\ 0 &, \hat{P}^i_J=0\end{cases},\]
	we further have
	\begin{equation}
	\tilde{P}^i_J\leq \hat{P}^i_J\leq (1+\delta)\cdot\tilde{P}^i_J\label{EqPhat2}.
	\end{equation}
	Combining \eqref{EqPhat} and \eqref{EqPhat2} yields
	\[\tilde{P}^i_J\leq P^i_J\leq (1+\delta)^i\cdot \tilde{P}^i_J\] as desired.\\
	A similar argument also applies to the cost-values: By the induction hypothesis, we get
	\begin{align*}
	\hat{C}^i_j&=\tilde{C}^{i-1}_j + c_i\cdot\beta^j\leq \underbrace{C^{i-1}_j + c_i\cdot\beta^j}_{=C^i_j}\leq(1+\delta)^{i-1}\cdot \tilde{C}^{i-1}_j + (1+\delta)^{i-1}\cdot c_i\cdot\beta^j\\
	&=(1+\delta)^{i-1}\cdot\hat{C}^i_j,
	\end{align*}
	\begin{equation}
	\text{so }\quad\hat{C}^i_j\leq C^i_j\leq (1+\delta)^{i-1}\cdot\hat{C}^i_j\label{EqChat}.
	\end{equation}
	As \[\tilde{C}^i_j=\begin{cases} (1+\delta)^{\lfloor\log_{1+\delta}\hat{C}^i_j\rfloor} &, \hat{C}^i_j>0\\ 0 &, \hat{C}^i_j=0\end{cases},\]
	\begin{equation}
	\text{we further have }\quad\tilde{C}^i_j\leq \hat{C}^i_j\leq (1+\delta)\cdot\tilde{P}^i_j\label{EqChat2}.
	\end{equation}
	As before, combining \eqref{EqChat} and \eqref{EqChat2} yields
	$\tilde{C}^i_j\leq C^i_j\leq (1+\delta)^i\cdot \tilde{C}^i_j$ as desired.
\end{proof}
\LemCloseCosts*
\begin{proof}
Let $\tilde{\mathcal{C}}=:(n,(\tilde{P}_J)_{J\subseteq [k]}, (\tilde{C}_j)_{j=1}^k)$ and $\mathcal{C}:=(n,(P_J)_{J\subseteq [k]}, (C_j)_{j=1}^k):=\mathrm{Cand}^n(B)$.
	Then Lemma~\ref{lemObjectiveWithJs} tells us that
	\[\ev_{\prod_{i=1}^n\mu_i}[c_B]=\mathrm{Cost}(\mathcal{C}).\]
	By Lemma~\ref{LemRatios}, we obtain
	\begin{align*}
	\mathrm{Cost}(\tilde{\mathcal{C}})&=\sum_{J:\emptyset\neq J\subseteq [k]}\tilde{P}_J\cdot\min_{j\in J} \tilde{C}_j\notag\\
	&\leq \underbrace{\sum_{J:\emptyset\neq J\subseteq [k]}P_J\cdot\min_{j\in J} C_j}_{=\mathrm{Cost}(\mathcal{C})} \notag\\
	&\leq \sum_{J:\emptyset\neq J\subseteq [k]}(1+\delta)^n\cdot\tilde{P}_J\cdot\min_{j\in J} (1+\delta)^n\cdot \tilde{C}_j \notag\\
	&=(1+\delta)^{2n}\cdot\mathrm{Cost}(\tilde{\mathcal{C}}).
	\end{align*}
	By our choice of $\delta=\frac{\epsilon}{4n}$, we obtain $(1+\delta)^{2n}\stackrel{\eqref{EqOneplusX}}{\leq}\eul^{2n\cdot\delta}=\eul^{\frac{\epsilon}{2}}\stackrel{\eqref{EqBoundExpX}}{\leq}1+\epsilon,$
	which concludes the proof.
\end{proof}
\begin{proposition}
	For each $\mathcal{C}\in\mathcal{T}$, we have $\mathrm{Witness}(\mathcal{C})\in\mathrm{AllWits}(\mathcal{C}).$\label{PropWitnessAmongAllWits}
\end{proposition}
\begin{proof}
Induction on $i$ such that $\mathcal{C}\in\mathcal{T}_i$. For $i=1$, the assertion is clear. Now, let $i\geq 2$ and pick $\mathcal{C}_i\in\mathcal{T}.$ Consider the last iteration of the foreach-loop in line~\ref{LineForLoopOldCands} where $\mathrm{Witness}(\mathcal{C}_i)$ is updated and let $B^i=(b^{i,j})_{j=1}^k$ be its final value. Let $\mathcal{C}_{i-1}\in\mathcal{T}$ be the candidate corresponding to this iteration of the foreach-loop, let $B^{i-1}=(b^{i-1,j})_{j=1}^{k}$ be the final witness of $\mathcal{C}_{i-1}$ and let $(\beta^j)_{j=1}^k$ be the vector for which $B^i$ arises from $B^{i-1}$, that is $\beta^j=b^{i,j}_i$ and $b^{i-1,j}_l=b^{i,j}_l$ for $l\neq i$. By the induction hypothesis, we have $B^{i-1}\in\mathrm{AllWits}(\mathcal{C}_{i-1})$, meaning that $B^i$ is added to $\mathrm{AllWits}(\mathcal{C}_i)$ in the iteration of the foreach-loop in line~\ref{LineForLoopAllWits}. This concludes the proof.
\end{proof}
\PropAllSolutionsOccur*
\begin{proof}
Induction on $i$. For $i=1$, this is clear. Let $i\geq 2$ and consider $B^i=(b^{i,j})_{j=1}^k$ such that $b^k=a^*$ and $b^j_l=0$ for $j=1\dots,k-1$ and $l=i+1,\dots,n$. Define $B^{i-1}=(b^{i-1,j})_{j=1}^k$ by $b^{i-1,k}=a^*$ and \[b^{i-1,j}_l:=\begin{cases} b^{i,j}_l &, l\leq i-1, j\leq k-1\\ 0 &, l\geq i, j\leq k-1\end{cases}.\] Then the induction hypothesis is applicable to $B^{i-1}$ and we know that there is $\mathcal{C}_{i-1}\in\mathcal{T}_{i-1}$ with $(B^{i-1,j})_{j=1}^k\in\mathrm{AllWits}(\mathcal{C}_{i-1})$. Let $(\beta^j)_{j=1}^k$ be given by $\beta^j=b^{i,j}_i$ and consider the iteration of the foreach-loop in line~\ref{LineForLoopAllWits} where $\mathcal{C}_{i-1}$ and $(\beta^j)_{j=1}^k$ are chosen in the outer loops and $B^{i-1}\in\mathrm{AllWits}(\mathcal{C}_{i-1})$ is considered. Then for the respective candidate $\mathcal{C}_i$ we are currently considering, $B^i$ is added to $\mathrm{AllWits}(\mathcal{C}_i)$.
\end{proof}
The remainder of this section deals with the analysis of the running time of Algorithm~\ref{AlgDPValidCandsRounded}. In doing so, Proposition~\ref{PropPossibleCostValues} provides a polynomial bound on the number of possible cost values occurring within a rounded candidate, while Proposition~\ref{PropPossiblePropValues} bounds the number of different probabilities we may encounter. Lemma~\ref{LemTPolynomial} combines these two results into a polynomial bound on the size of $\mathcal{T}$.
\begin{proposition}
	Let an instance of the discrete Pareto cover problem be given and let $(m,(P_J)_{J \subseteq [k]}, (C_j)_{j=1}^k)\in \mathcal{T}$. Then for all $j\in[k]$, we either have $C_j=0$, or $C_j=(1+\delta)^l$ with $l\in [\log_{1+ \delta}(a_1)+\min_{i=1}^n \log_{1+\delta} (c_i)-n, \log_{1+\delta}(\sum_{i=1}^n c_i)]\cap \mathbb{Z}$.\label{PropPossibleCostValues}
\end{proposition}
\begin{proof}
	By Lemma~\ref{LemRatios}, there is a valid candidate with witness $B=(b^j)_{j=1}^k$ and costs $(\hat{C}_j)_{j=1}^k$ such that $C_j\leq \hat{C}_j\leq (1+\delta)^{m}\cdot C_j\leq (1+\delta)^n\cdot C_j$ for all $j$. If $\hat{C}_j=0$, this implies $C_j=0$. Otherwise, we know that $0<\hat{C}_j=\sum_{i=1}^n c_i\cdot b^j_i$, meaning that there is $i\in[n]$ such that $b^j_i>0$ and, hence, $b^j_i\geq a_1$.\\ By non-negativity of costs, this implies $a_1\cdot\min_{i=1}^n c_i\leq \hat{C}_j\leq \sum_{i=1}^n c_i\cdot 1,$ and, hence, $(1+\delta)^{-n}\cdot a_1\cdot\min_{i=1}^n c_i\leq C_j\leq \sum_{i=1}^n c_i.$ As $C_j$ has to be an integer power of $(1+\delta)$, applying logarithms yields the claim.
\end{proof}
\begin{proposition}
	Let an instance of the discrete Pareto cover problem be given and let $(m,(P_J)_{J \subseteq [k]}, (C_j)_{j=1}^k)\in \mathcal{T}$. Then for all $J\subseteq[k]$, we either have $P_J=0$, or $P_J=(1+\delta)^l$ with $l\in [n\cdot\min_{i=1,\dots,n, l=0,\dots,M+1, p^i_l > 0}\log_{1+\delta}(p^i_l) -n, 0]\cap \mathbb{Z}$.\label{PropPossiblePropValues}
\end{proposition}
\begin{proof}
	By Lemma~\ref{LemRatios}, there is a valid candidate with witness $B=(b_j)_{j=1}^k$ and probabilities $(\hat{P}_J)_{J\subseteq[k]}$ such that $P_J\leq \hat{P}_J\leq (1+\delta)^{n}\cdot P_J$ for all $j$. If $\hat{P}_J=0$, this implies $P_J=0$. Otherwise, we know that\[0<\hat{P}_J=(\prod_{i=1}^n\mu_i)(\{x\in [0,1]^n: J^m(x)=J\})=\sum_{\substack{t\in \{0,\dots,M+1\}^n:\\J^m((a_{t_i})_{i=1}^n)=J}}\prod_{i=1}^n p^i_{t_i}.\] Hence, there has to be $t\in \{0,\dots,M+1\}^n$ such that $p^i_{t_i}>0$ for $i=1,\dots,n$, and  \[(\min\{p^i_l: i=1,\dots,n, l=0,\dots,M+1, p^i_l > 0\})^n\leq\prod_{i=1}^n p^i_{t_i}\leq\hat{P}_J\leq 1,\] where the last inequality follows from the fact that $(\prod_{i=1}^n\mu_i)$ is a probability measure. Consequently, \[(1+\delta)^{-n}\cdot (\min\{p^i_l: i=1,\dots,n, l=0,\dots,M+1, p^i_l > 0\})^n\leq P_J\leq 1.\] As $P^J$ is an integer power of $(1+\delta)$, applying logarithms yields the claim.
\end{proof}
\LemTPolynomial*
\begin{proof}
Concerning the bound on $\mathcal{T}$, observe that the first factor is the number of possible values for $(P_J)_{J\subseteq [k]}$, the second factor is the number of possible values for $(C_j)_{j=1}^k$ and the last factor is the number of possible values for $i$.
	
	 To show that the upper bound on $\mathcal{T}$ is polynomial, note that for constant $k$, $k$ and $2^k$ are constant, so it suffices to verify that each of the terms $n$,\[-n\cdot\min_{\substack{i=1,\dots,n,\\ l=0,\dots,M+1,\\ p^i_l > 0}}\log_{1+\delta}(p^i_l)+n+2\] and $n+2+\log_{1+\delta}(\sum_{i=1}^n c_i)-\log_{1+ \delta}(a_1)-\min_{i=1}^n \log_{1+\delta} (c_i))$ is polynomial in the encoding length of the input instance. Given that for $x\in\mathbb{R}_{>0}$, we have $\log_{1+ \delta}(x)=\frac{\log_2(x)}{\log_2(1+\delta)}$, it remains to check that $\frac{1}{\log(1+\delta)}$ is polynomial in the input size and $\epsilon^{-1}$. Indeed, \[\log_{2}(1+ \delta)=\frac{\ln(1+\delta)}{\ln(2)}\geq \ln(1+\delta)=\ln(1)+\int_{1}^{1+\delta}\frac{1}{x}dx\geq 0+\int_{1}^{1+\delta}\frac{1}{2}dx=\frac{\delta}{2}=\frac{\epsilon}{8n},\]
	so $\frac{1}{\log(1+\delta)}\leq 8\cdot n\cdot \epsilon^{-1}$, which concludes the proof.
\end{proof}
\TheoPolyTime*
\begin{proof} 
By Lemma~\ref{LemTPolynomial}, there is a polynomial number of iterations of the foreach-loop in line~\ref{LineForLoopOldCands}. Moreover, each time, we perform at most $(M+2)^k$ iterations of the loops in line~\ref{LineForeachBeta1} and line~\ref{LineForBeta2}, which is polynomial in the input size as well.
	The candidate updates in lines~\ref{LineStartInner1}-\ref{LineEndInner1} and lines~\ref{LineStartInner2}-\ref{LineEndInner2} run in time $\mathcal{O}(4^k)$, assuming that the computation of logarithms and exponentiation can be done in constant time. If not, by the bounds we have obtained on the range of possible values, we could alternatively compute a polynomial-sized lookup-table in polynomial time. Finally, all terms of the form $\mu_i((a,b])$ that occur during the algorithm can be evaluated in linear time by simply summing up the probabilities $p^i_l$ for all $l$ such that $a<a_l\leq b$.
\end{proof}
\TheoApproxGuarantee*
\begin{proof}
Pick an optimum solution $B^*=(b^{*,j})_{j=1}^k$ with $b^{*,k}=a^*$. Denote the solution we return by $B=(b^j)_{j=1}^k$. Let further $\tilde{\mathcal{C}}:=(n,(\tilde{P}_J)_{ J\subseteq[k]},(\tilde{C}_j)_{j=1}^k)$ be the candidate of minimum cost we choose, let $\mathcal{C}:=(n,P_J)_{J\subseteq[k]},(C_j)_{j=1}^k):=\mathrm{Cand}^n(B)$, and let \[\hat{\mathcal{C}}:=(n,(\hat{P}_J)_{J\subseteq[k]},(\hat{C}_j)_{j=1}^k)\in\mathcal{T}_n\] such that $B^*\in\mathrm{AllWits}(\hat{\mathcal{C}})$, which exists by Proposition~\ref{PropAllSolutionsOccur}.
	By Lemma~\ref{lemObjectiveWithJs}, we know that the objective value attained by $B$ equals \begin{equation}
	\ev_{\prod_{i=1}^n\mu_i}[c_B]=\mathrm{Cost}(\mathcal{C})\label{EqCostAlg}.
	\end{equation}
	By Lemma~\ref{LemCloseCosts} and Proposition~\ref{PropWitnessAmongAllWits}, we further obtain
	\begin{equation}
	\mathrm{Cost}(\mathcal{C})\leq (1+\epsilon)\cdot \mathrm{Cost}(\tilde{\mathcal{C}}), \label{EqCostAlg2}
	\end{equation}
	as well as
	\begin{equation}
	\mathrm{Cost}(\hat{\mathcal{C}})\leq \ev_{\prod_{i=1}^n\mu_i}[c_{B^*}].\label{EqCostOpt2}
	\end{equation}
	Finally, our choice of $\tilde{\mathcal{C}}$ implies that
	\begin{equation}
	\mathrm{Cost}(\tilde{\mathcal{C}})\leq \mathrm{Cost}(\hat{\mathcal{C}})\label{EqMinCost}.
	\end{equation}
	All in all, this yields
	\begin{align*}\mathrm{Cost}(B)&\stackrel{\eqref{EqCostAlg}}{=}\mathrm{Cost}(\mathcal{C})\stackrel{\eqref{EqCostAlg2}}{\leq} (1+\epsilon)\cdot\mathrm{Cost}(\tilde{\mathcal{C}})\stackrel{\eqref{EqMinCost}}{\leq}(1+\epsilon)\cdot\mathrm{Cost}(\hat{\mathcal{C}})\\&\stackrel{\eqref{EqCostOpt2}}{\leq}(1+\epsilon)\cdot\ev_{\prod_{i=1}^n\mu_i}[c_{B^*}]=(1+\epsilon)\cdot\mathrm{OPT}. \end{align*}
\end{proof}

\end{document}